\documentclass[12pt]{amsart}

\usepackage{color}

\usepackage{CJK}

\newtheorem{theorem}{Theorem}[section]
\newtheorem{lemma}[theorem]{Lemma}
\newtheorem{proposition}[theorem]{Proposition}
\newtheorem{corollary}[theorem]{Corollary}

 \theoremstyle{definition}

\theoremstyle{remark}
\newtheorem{remark}[theorem]{Remark}

\theoremstyle{claim}

\theoremstyle{fact}

\numberwithin{equation}{section}

\begin{document}
\begin{CJK*}{GBK}{song}

\title[Fully non-linear elliptic equations]
%{The Dirichlet problem for Degenerate Fully non-linear Elliptic Equations on  Hermitian manifolds}
%{On the gradient estimate and regularity for degenerate fully non-linear elliptic equations on  Hermitian manifolds with Levi flat boundary}
{Regularity of  fully non-linear elliptic equations on Hermitian manifolds}

%{A quantitative boundary estimate and regularity of  fully non-linear elliptic equations on Hermitian manifolds}

\author{Rirong Yuan}
%\address{School of Mathematics (Zhuhai), Sun Yat-sen University, Zhuhai 519082, China}
\email{rirongyuan@stu.xmu.edu.cn}
%\email{yuanrr@mail.sysu.edu.cn;    rirongyuan@hotmail.com}
%\thanks{The author is partially supported by the NSFC (Grant No. 11801587)}
%\thanks{The author is partially supported by National Natural Science Foundation of China (Grant No. 11801587)}
\thanks{This is the first part of a series of   researches devoted to the study of Dirichlet problem for fully non-linear elliptic equations on complex manifolds, which include [arXiv:2001.09238], [arXiv:2106.14837], [Pure Appl. Math. Q. {\bf 16} (2020),   1585-1617; MR4221006] and [Calc. Var. PDE. {\bf 60} (2021),  Paper No. 162, 20 pp.; MR4290375]. 
	%Research supported in part by the National Natural Science Foundation of China (Grant No. 11801587)
}
%\thanks{}
\date{}

\maketitle
% {\centering\footnotesize Dedicated to my family \par}

\begin{abstract} 

In this paper we propose new insights and ideas  to set up  quantitative  boundary estimates
for solutions to Dirichlet problem of a class of  fully non-linear elliptic equations on compact
Hermitian manifolds  with real analytic Levi flat boundary.  With the quantitative boundary estimates at hand,
we can establish the gradient estimate and give a unified approach to investigate the
 existence and regularity of solutions 
 %for such fully non-linear elliptic equations 
 of Dirichlet problem with
sufficiently smooth boundary data, which include the geodesic equation in the space of K\"ahler
metrics as a special case.  Our method can also be applied to Dirichlet problem for analogous
fully non-linear elliptic equations on a compact Riemannian manifold with concave boundary.

\end{abstract}

\maketitle

\section{Introduction}

\medskip

%\subsection{Background}

%The study of fully non-linear elliptic equations generated by symmetric functions of the eigenvalues of the complex Hessian
%on complex manifolds
 %has been studied %extensively 
 %in literature, as theinterests from PDEs, complex geometry and mathematical physics.
The most important fully non-linear elliptic equation 
 is perhaps the complex Monge-Amp\`{e}re equation, 
%corresponding to $f(\lambda)=  \prod_{i=1}^n \lambda_i^{1/n}$,
%since complex Monge-Amp\`{e}re operator 
which is closely related to the volume form and the
 representation of Ricci form of K\"{a}hler manifolds.
% This fundamental problem claims that
% any representative of first Chern class can be written as the Ricci curvature of a unique K\"ahler metric
 %whose K\"ahler form is
 %, which    is called
 Calabi conjectured in  \cite{Calabi1954} that any $(1,1)$-form which represents
  the first Chern class  of a  closed K\"ahler manifold   is the Ricci form
 of a K\"ahler metric with the same K\"ahler class as the original one.
 %whose K\"ahler form lies in the same K\"ahler class.
% whose  is any given representative of the first Chern class of this corresponding K\"ahler manifold.
 This fundamental problem   can be reduced to
solving  non-degenerate and smooth
 complex Monge-Amp\`{e}re equation
  on closed K\"{a}hler manifolds.
   In his pioneering paper,  Yau   \cite{Yau78} proved Calabi's conjecture and  showed that
  %   any representative of first Chern class can be written as the Ricci curvature of a K\"ahler
     any closed K\"{a}hler manifold with   zero or negative first Chern class
     admits a unique  K\"{a}hler-Einstein metric in the given K\"{a}hler class.
   % Moreover, %in the same paper,
   %  Yau also studied the complex Monge-Amp\`ere equation with meromorphic right-hand side.
 %Yau's work has a tremendous  influence on PDEs, algebraic geometry, complex geometry and  mathematical physics.
 %  (cf. \cite{C-H-S-W1985,Gross-Huybrechts-Joyce,Yau77,Yau-Collected-2}).
The existence of K\"{a}hler-Einstein metrics on  closed K\"ahler manifolds of  negative
first Chern class was also  proved  by Aubin \cite{Aubin1976Equations} independently.
 A heat equation proof of Yau's theorem was obtained by  Cao  \cite{Cao-1985-KRflow} in which he introduced
the K\"ahler-Ricci flow.
 %Recently, 
 Yau's work was partially  extended by Tosatti-Weinkove \cite{Tosatti2010Weinkove}
 to closed Hermitian manifolds (see also   \cite{Gill2011} for a parabolic version).
 The other seminal  works concerning the complex Monge-Amp\`{e}re equation was done by Bedford-Taylor \cite{Bedford1976Taylor,Bedford1982Taylor}
 on generalized solutions in the sense of pluripotential theory,
 and by Caffarelli-Kohn-Nirenberg-Spruck \cite{Caffarelli1985The} on
 the Dirichlet problem for the complex Monge-Amp\`{e}re equation on  a strictly pseudoconvex domain in $\mathbb{C}^n$.
 % Caffarelli-Kohn-Nirenberg-Spruck's  work was further extended by B. Guan \cite{Guan1998The}
  %   to general bounded domains $\Omega\subset \mathbb{C}^n$,
% in which the author used the \textit{admissible} subsolution  satisfying \eqref{existenceofsubsolution}
%   to relax the geometric restriction to $\partial \Omega$.F
  %The  standpoint of the subsolution, with no curvature assumptions on the boundary %of the underlying space
% being made,  was used by Hoffman-Rosenberg-Spruck \cite{Hoffman1992Boundary}
% and further by Guan-Spruck  \cite{Guan1993Boundary} in real variables, and has great advantage
%in applications to geometric problems, see for instance,  \cite{Chen,GuanP2002The,Guan2009Zhang}.
%It would be worthwhile to note that the Dirichlet problem is not always solvable without the 
% assumption of the existence of  a subsolution.
%%There have also been increasing interests from complex geometry

This paper is devoted to investigating the solvability and regularity of solutions to fully non-linear elliptic equations on Hermitian manifolds.
Let $(M,J,\omega)$ be  a compact Hermitian manifold
 of complex dimension $n\geq 2$ with Levi flat 
 boundary $\partial M$,  where   $\omega$ is
 the  K\"ahler form %which is
 locally given by  $$\omega=\sqrt{-1}g_{i\bar j}dz^i\wedge d\bar z^j,$$
 and  $J$ denotes the underlying complex structure. Throughout the paper, unless otherwise indicated, we assume in addition that the boundary is smooth and real analytic. 
 % (i.e. the Levi form of $\partial M$ vanishes identically on $\partial M$), 
 % (in the sense that near each point on boundary, there are local holomorphic coordinates  $(z_1,\cdots, z_n)$ 
  %such that $\partial M$ is given by $\mathfrak{Re}(z_n)=0$)
 %(following \cite{PhongSongSturm-CMA}  
  %we say $\partial M$ is real analytic Levi flat if 
  %one can pick local holomorphic coordinates  $(z_1,\cdots, z_n)$ 
  %such that $\partial M$ is locally of the form $\mathfrak{Re}(z_n)=0$),
Let $\chi=\sqrt{-1}\chi_{i\bar j}dz^i\wedge d\bar z^j$ be a   smooth real $(1,1)$-form on $\bar M:= M\cup \partial M$.
  %The increasing interest motives us to
We are concerned with   fully non-linear elliptic equations  for deformation of $\chi$,
 \begin{equation}
\label{mainequ1}
\begin{aligned}
F(\mathfrak{g}[u]):=
f(\lambda(\mathfrak{g}[u]))=  \psi    \mbox{ in } M,
%\end{aligned} \end{equation}
% \begin{equation}\label{mainequ2}\begin{aligned}
%u =\,& \varphi \,&  \mbox{ on } \partial   M,
\end{aligned}
\end{equation}
with the boundary value condition $u =  \varphi  \mbox{ on } \partial M,$
where   $\psi$ and $\varphi$ are both two functions with suitable regularity, $\mathfrak{g}[u]=\chi+\sqrt{-1}\partial\overline{\partial} u$,
 and $\lambda(\mathfrak{g}[u])=(\lambda_1,\cdots,\lambda_n)$ denote the eigenvalues of $\mathfrak{g}[u]$ with respect to
 $\omega$.
 
  %Throughout the paper, unless otherwiseindicated, we assume in addition that the boundary is smooth and real analytic. 
 % Cartan's theorem  implies that the real analytic Levi flat boundary  is always locally given by $\mathfrak{Re}(z_n)=0$ in suitable holomorphic coordinates (cf. \cite{Cartan-1933,Lebl-Fern}). 
 % Indeed such local holomorphic coordinate system  is only needed in the proof  of Proposition \ref{Tangential-Normal derivatives}.
  % in the sense that near each point on boundary there are local holomorphic coordinates  $(z_1,\cdots, z_n)$ 
 % such that $\partial M$ is given by $\mathfrak{Re}(z_n)=0$ (cf.\cite{Blocki2009Onthegeodesic,Phong2009Sturm,PhongSongSturm-CMA}).
%If the Levi flat boundary is real analytic  then it is  locally flat according to Cartan theorem (cf. \cite{Lebl-Fern}).

  The study of fully non-linear elliptic equations analogous to
  \eqref{mainequ1} in the setting of real variables  goes back to the work of
  Caffarelli-Nirenberg-Spruck  \cite{CNS3} and
Ivochkina \cite{Ivochkina1981} who dealt with some special cases.
The function $f$ is assumed to be a smooth symmetric function  defined in an open
symmetric  and convex cone $\Gamma\subset \mathbb{R}^n$
with vertex at the origin, and boundary $\partial \Gamma\neq \emptyset$,
$$\Gamma_{n} \subseteq \Gamma \subset \Gamma_1,$$
where $\Gamma_n=\{\lambda\in \mathbb{R}^{n}: \mbox{ each } \lambda_i>0\} \mbox{ and } \Gamma_1=\{\lambda\in \mathbb{R}^{n}: \sum_{i=1}^n \lambda_i>0\}.$
Moreover,  $f$ shall satisfy the following structure conditions:
%Following  Caffarelli-Nirenberg-Spruck \cite{CNS3}
\begin{equation}
\label{elliptic}
 \begin{aligned}
 f_{i} := \frac{\partial f}{\partial \lambda_{i}}> 0  \mbox{ in } \Gamma,\  1\leq i\leq n,
  \end{aligned}
\end{equation}
\begin{equation}
 \begin{aligned}
\label{concave}
 f \mbox{ is concave in } \Gamma,
 \end{aligned}
\end{equation}
%and the non-degenerate condition
\begin{equation}
 \label{nondegenerate}
  \begin{aligned}
   \delta_{\psi,f}:=\inf_{\bar M} \psi -\sup_{\partial \Gamma} f>0,
 \end{aligned}
 \end{equation}
where   $$\sup_{\partial \Gamma}f := \sup_{\lambda_{0}\in \partial \Gamma } \limsup_{\lambda\rightarrow \lambda_{0}} f(\lambda).$$
%As in \cite{CNS3}  conditions \eqref{elliptic}-\eqref{nondegenerate}  play fundamental roles in the study of   equation \eqref{mainequ1}.
The constant $\delta_{\psi, f} $  measures if the equation is degenerate. Namely,
the equation is   called non-degenerate (respectively, degenerate) if  $\delta_{\psi,f}>0$ (respectively, if $\delta_{\psi,f}$ vanishes).

In addition, we assume
 \begin{equation}
 \begin{aligned}
 \label{addistruc}
  \mbox{For any } \sigma<\sup_{\Gamma} f 
   \mbox{ and  } \lambda\in \Gamma, \quad \lim_{t\rightarrow + \infty} f(t\lambda)>\sigma
 \end{aligned}
\end{equation}
so that one can apply  Sz\'ekelyhidi's  \cite{Gabor}   Liouville type theorem
 extending a result of Dinew-Ko{\l}odziej \cite{Dinew2012Kolodziej}
to derive the gradient estimate via a blow-up argument.
We shall emphasize that  conditions \eqref{elliptic}, \eqref{concave} and  \eqref{addistruc}
% is satisfied by many natural functions, for instance  $f$ with $\Gamma=\Gamma_n$ or $f$
 allow  many natural functions, for instance if  $f$
is homogeneous of degree one with $f>0$ in $\Gamma$, or
 the function $f$ associated with $\Gamma=\Gamma_n$.
% We  give in Appendix \ref{Appendix}
%a characterization of the level sets of $f$   satisfying  \eqref{elliptic}, \eqref{concave} and  \eqref{addistruc},
%which plays an important role in proving the boundary estimates for `tangential-normal' derivatives.

 In order to study equation \eqref{mainequ1} within the framework of elliptic equations,
we shall find the solutions in the class of
$C^2$-\textit{admissible} functions $u$  satisfying 
$$\lambda(\mathfrak{g}[u])\in \Gamma.$$ 

%  The study of such  equations can be traced back to Ivochkina \cite{Ivochkina1981}
%who dealt with some special cases, and  to  Caffarelli-Nirenberg-Spruck \cite{CNS3}
%on the Dirichlet problem  for certain fully non-linear elliptic equations on the bounded domains $\Omega\subset \mathbb{R}^n$.
%Since then, such fully non-linear elliptic equations have been studied extensively,
 %see  \cite{Guan12a,IvochTrudingerWang04,LiYY1990,Gabor,Trudinger95} and the references therein.
%We derive the gradient estimate by using a  quantitative estimate for complex Hessian of the solution, and solve
%the Dirichlet problem for  fully non-linear elliptic  equations.   %of the form
%dating back to  Caffarelli-Nirenberg-Spruck \cite{CNS3},
%, and Ivochkina's  work concerning  some  special cases \cite{Ivochkina1981},
%
%Such  fully non-linear elliptic (or parabolic) equations have been studied extensively,
  %as the increasing interests from  complex geometry  and mathematical physics.
%(cf. \cite{CollinsJacobYau,Fu2010ICM,Fu2007Yau,Fu2008Yau,Gauduchon84,Gill-2011,guan-nie,GTW15,Tosatti2013Weinkove,Tosatti2017Weinkove}
%and  references therein).
%One of the most important geometric partial differential equations   is  complex Monge-Amp\`{e}re equation,
%which is closely related to  Calabi conjecture in  complex   geometry  \cite{Calabi1954},
%
%Inspired  by a serious of works \cite{Caffarelli1985The,CNS1,CNS3}
In the theory of fully non-linear elliptic equations, the
standpoint of a subsolution $\underline{u}\in C^2(\bar M)$ satisfying
\begin{equation}
\label{existenceofsubsolution}
%\left\{
\begin{aligned}
 f(\lambda(\mathfrak{g}[\underline{u}]) )
 \geq
 \psi   \mbox{ in }   M,
\end{aligned}
%\right.
\end{equation}
with $\underline{u}=   \varphi \mbox{ on } \partial  M,$
 without geometric restriction to the boundary, plays important roles in deriving \textit{a priori} estimates for Dirichlet problem and  has a great advantage
in  applications to geometric problems (cf. \cite{Chen,Guan1993Boundary,GuanP2002The,Guan2009Zhang,Hoffman1992Boundary} and references therein).
  The \textit{admissible} subsolution was used by Hoffman-Rosenberg-Spruck \cite{Hoffman1992Boundary}
 and further developed in \cite{Guan1993Boundary,Guan1998The} to derive second order boundary estimates for real and complex Monge-Amp\`ere equation on general bounded domains.
In addition,   Guan \cite{Guan12a} recently used it to derive the global second order estimate for  general fully non-linear elliptic
equations on Riemannian manifolds.
Influenced by the work of Guan \cite{Guan12a},
a flexible   notion of a  $\mathcal{C}$-subsolution was recently introduced by Sz\'ekelyhidi  \cite{Gabor}.
 %With assuming the existence of a $\mathcal{C}$-subsolution,
     With  the assumption of the existence   of  $\mathcal{C}$-subsolutions,
 Sz\'ekelyhidi established $C^{2,\alpha}$-estimate for  \textit{admissible}
 solutions of fully non-linear elliptic equations satisfying \eqref{elliptic}-\eqref{addistruc} on closed Hermitian manifolds.
The concept of $\mathcal{C}$-subsolution turns out to be   applicable  for fully non-linear elliptic equations in the setting of closed manifolds (cf. \cite{GTW15,CollinsJacobYau}).

\subsection{Statement of  main results}
We show that  Dirichlet problem  \eqref{mainequ1}   is uniquely  solvable
   in the class of \textit{admissible} functions, provided %in addition
   that the function $\psi$ and  boundary data $\varphi$ are sufficiently smooth,  and
   the Dirichlet problem  admits an \textit{admissible} subsolution taking the same boundary data.

%Our main theorems can be stated as follows:

%\begin{theorem}
%\label{existence}
%Let $\sigma$ denote   the distance function to $\partial M$.
%Let  $\psi\in C^{k,\alpha}(\bar M)$ and $\varphi\in C^{k+2,\alpha}(\partial M)$, $k\geq 2$, $\forall 0<\alpha<1$.
%In addition to \eqref{elliptic}-\eqref{addistruc}, we assume
% that  there exists an admissible subsolution $\underline{u}\in C^{2}(\bar M)$  such that
%\begin{equation}
%\label{existenceofsubsolution}
%\left\{
%\begin{aligned}
% f(\lambda(\mathfrak{g}[\underline{u}]) )
% \geq
%\,&
% \psi   \,&\mbox{ in }   M, \\
%\underline{u}= \,& \varphi    \,&\mbox{ on } \partial  M.
%\end{aligned}
%\right.
%\end{equation}
%Then Dirichlet problem   \eqref{mainequ1}
%has a unique  admissible solution $u\in C^{k+2,\alpha}(\bar M)$.
% Moreover,
%\begin{equation}
%\begin{aligned}
%\label{boundaryestimate1233}
%|u|_{C^{2}(\bar M)}\leq C,
%  \end{aligned}
%\end{equation}
%  where $C$ depends only on $|\sigma|_{C^{1,1}(\partial M)}$,
%$|\varphi|_{C^{2,1}(\bar M)}$, $|\psi|_{C^{1,1}(\bar M)}$, $|\underline{u}|_{C^{1,1}(\bar M)}$,
%$|\chi|_{C^{1,1}(\bar M)}$  and other known data under control.
%\end{theorem}

\begin{theorem}
\label{existence}
Let $(M,J,\omega)$ be a compact Hermitian manifold with real analytic Levi flat boundary.
Let  $\psi\in C^{k,\alpha}(\bar M)$ and $\varphi\in C^{k+2,\alpha}(\partial M)$ for an
integer $k\geq 2$ and $0<\alpha<1$.
In addition to \eqref{elliptic}-\eqref{addistruc}, we assume
 that  there exists an admissible subsolution $\underline{u}\in C^{3}(\bar M)$ obeying \eqref{existenceofsubsolution}.
Then Dirichlet problem   \eqref{mainequ1}
has a unique  admissible solution $u\in C^{k+2,\alpha}(\bar M)$. Moreover,
\begin{equation}
\begin{aligned}
\label{boundaryestimate1233}
|u|_{C^{2}(\bar M)}\leq C,
\end{aligned}
\end{equation}
where $C$ depends only on
$|\varphi|_{C^{2,1}(\bar M)}$, $|\psi|_{C^{1,1}(\bar M)}$, $|\underline{u}|_{C^{2}(\bar M)}$,
 $|\chi|_{C^{2}(\bar M)}$  %$\partial M$ up to their second derivatives
  and other known data (but not on $(\delta_{\psi,f})^{-1}$).
\end{theorem}
 %We also let $\sigma$ denote the distance function to $\partial M$.
 \begin{remark}
Throughout this paper,
%for the constant $\delta_{\psi,f}$ defined on \eqref{nondegenerate},
 we say a constant $C$ does not depend  on $(\delta_{\psi,f})^{-1}$ if $C$
remains  bounded as $\delta_{\psi,f}$  tends to zero. We say the constant $\kappa$ depends on $\delta_{\psi,f}$
 if $\kappa\rightarrow 0$ when $\delta_{\psi,f}\rightarrow 0$.
 %  where  $\delta_{\psi,f}$ is the constant  which measures  if the equation is  degenerate.
 %The equation is   called degenerate if  $\delta_{\psi,f}$ vanishes.
 Moreover, in the theorems,
we assume  that the function $\varphi$ is
extended to a $C^3$ function on $\bar M$, still denoted as $\varphi$.
\end{remark}

We can apply our results to investigate   degenerate equations, as our  \textit{a priori}  estimates do not depend on $(\delta_{\psi,f})^{-1}$.
  The degenerate fully non-linear elliptic equations have  been studied 
  %from the point of view of PDEs 
  %in different scenarios,
 from different aspects, including   \cite{GuanP1997Duke,GuanP1999TrudingerWang,IvochTrudingerWang04,Krylov1984,Trudinger1987de,Trudinger95}. They  are also closely connected with 
 the research of certain %geometric
   objects in
 %certain problems from 
 differential geometry and analysis.
%There are some related problems and results
%  To the best of the author's knowledge,
For instance, the homogeneous of complex Monge-Amp\`ere equation %(HCMA)
\begin{equation}
\label{HCMA1}
\begin{aligned}
(\omega+\sqrt{-1}\partial \overline{\partial} u)^n=0 \mbox{ in } M, \quad   u=\varphi \mbox{ on } \partial M
\end{aligned}
\end{equation}
%where   $\omega_u=\omega+\sqrt{-1}\partial \overline{\partial} u$,
plays some important  roles in complex geometry and complex analysis
(cf. \cite{Bedford1979Taylor,ChernLevineNirenberg,Donaldson99,Mabuchi87,Semmes92}).
% See also Phong-Song-Sturm and Ross-Witt Nystr\"om's surveys \cite{PhongSongSturm-CMA,Ross2017Nystrom} or  Guedj-Zeriahi's  book \cite{Guedj2017Zeriahi}.
 See also  \cite{Guedj2017Zeriahi,PhongSongSturm-CMA,Ross2017Nystrom}.
 The following %topics and 
 works should be  mentioned:

 \begin{itemize}
  \item  Guan's \cite{GuanP2002The,GuanP2010Remark}  proof of Chern-Levine-Nirenberg conjecture on intrinsic norms.
 \item   The work concerning Donaldson's conjecture in K\"{a}hler geometry due to Chen \cite{Chen}.
 For more  related topics,
  %especially the construction and   regularity of  geodesic rays associated to test configurations,
   please refer   to
\cite{Blocki2009Onthegeodesic,Calamai2015Zheng,Chen2009III,Chen-Sun2014,Chen2008Tian,Donaldson2002Ho,HeWeiyong2015,Lempert2013Vivas,Phong2006Sturm,Phong2007JSMSturm,Phong2009Sturm}  and references therein.
 %\item   The characterization of totally real submanifolds by solutions of the homogeneous Monge-Amp\`ere equation (cf. \cite{Donaldson2002Ho,Semmes92,Chen2008Tian,Guan2010Li} and references therein).
 \item  The regularity of  pluricomplex Green function in the pluripotential theory (cf. \cite{Blocki05pluricomplex,demailly1987pluricomplex,Guan1998The,Guan2007pluricomplex,lempert1981,Lempert1983MA}
 and references therein).
\end{itemize}

Applying Theorem \ref{existence} and the method of approximation,
we can solve degenerate fully non-linear equations on compact Hermitian manifolds with real analytic Levi flat boundary.
%which enable us to attack various problems arising from complex geometry.

\begin{theorem}
\label{existencede}
%Suppose $(M, J, \omega)$ is a compact Hermitian manifold with $C^3$ Levi flat boundary $\partial M$.
%Suppose, in addition to \eqref{elliptic}-\eqref{concave}   and \eqref{addistruc}, that  $\partial M$ is of $C^3$.
%Suppose  \eqref{elliptic}, \eqref{concave}   and \eqref{addistruc} hold.
In addition to \eqref{elliptic}, \eqref{concave}   and \eqref{addistruc}, 
we assume $f\in C^\infty(\Gamma)\cap C^0(\bar \Gamma)$.
 Let $\varphi\in C^{2,1}(\partial M)$ and  $\psi\in C^{1,1}(\bar M)$ be a function satisfying
  $\delta_{\psi, f}=0$.
Suppose   that there is a strictly admissible subsolution $\underline{u}\in C^{2,1}(\bar M)$ with $\underline{u}=  \varphi    \mbox{ on } \partial  M$
satisfying
\begin{equation}
\label{stricsubsolution}
%\left\{
\begin{aligned}
 f(\lambda(\mathfrak{g}[\underline{u}]) )
 \geq
 \psi+\delta_{0}  \mbox{ in }   M
%\underline{u}= \,& \varphi \,&  \mbox{ on } \partial  M
\end{aligned}
%\right.
\end{equation}
for some   $\delta_0>0$.   %where $\varphi\in C^{2,1}(\partial   M)$ and $\psi\in C^{1,1}(\bar M)$ satisfies  $\delta_{\psi, f}\geq 0$.
Then Dirichlet problem \eqref{mainequ1}
%\begin{equation}
%\label{degenerate equa 1}
%%\left\{
%\begin{aligned}
%f(\lambda(\mathfrak{g}[u]))\,&= \psi \mbox{  }
% \,&\mbox{ in } M,  \\
%u \,&=\varphi  \,& \mbox{ on }\partial  M
%\end{aligned}
%%\right.
%\end{equation}
 admits a weak solution $u\in C^{1,\alpha}(\bar M)$, $\forall 0<\alpha<1$,
 with $\lambda(\mathfrak{g}[u])\in \bar \Gamma$  and $\Delta u \in L^{\infty}(\bar M)$.

\end{theorem}

Our method and results may enable us to attack various problems %arising 
from complex geometry, for instance
%As  a consequence,  our method can provide  approaches to
the geodesic equations in the space of
K\"ahler metrics
and the construction and regularity of
  geodesic rays associated to test configurations studied in   the cited literature above.

The primary difficulty %in this paper
 is deriving the gradient estimate.
It is pretty hard to prove the gradient bound for general fully non-linear elliptic equations
%it has been an open problem to prove the gradient bound   for general fully non-linear elliptic equations
 on curved complex  manifolds.
 %especially on the underlying space with boundary.
%By contrast with the Riemannian setting,  it is much more hard to
  %establish the gradient estimate for general fully non-linear elliptic equations on Hermitian manifolds.
 % One of the underlying reasons is the two different types of complex derives.
%We can also derive the gradient estimate by using a blow-up argument.
The blow-up argument is   an alternative approach to deriving gradient estimate, as shown
by Chen \cite{Chen} for Dirichlet problem of complex Monge-Amp\`ere equation on $X\times A$
where    $A=\mathbb{S}^1\times [0,1]$ and $X$ is a closed K\"ahler manifold, and by Dinew-Ko{\l}odziej  \cite{Dinew2012Kolodziej}
for  complex $k$-Hessian equation on  closed K\"ahler manifolds  using Hou-Ma-Wu's
\cite{HouMaWu2010} second order estimate of the form
\begin{equation}
\label{good1}
\begin{aligned}
\sup_{M}|\partial\bar \partial  u|\leq C (1+\sup_{ M} |\nabla u|^2).
\end{aligned}
\end{equation}
 %With  the assumption of the existence   of  $\mathcal{C}$-subsolutions,
The results of Hou-Ma-Wu and Dinew-Ko{\l}odziej in \cite{HouMaWu2010,Dinew2012Kolodziej} have been extended extensively by
 Sz\'ekelyhidi \cite{Gabor} %established second order estimate of the form \eqref{good1} for %the \textit{admissible} solutions solving
 to fully non-linear elliptic equations
    satisfying \eqref{elliptic}-\eqref{addistruc}   on closed Hermitian manifolds.
  % and then used them to derive gradient estimate for the \textit{admissible} solutions solving fully non-linear elliptic equations
%    satisfying \eqref{elliptic}-\eqref{addistruc}   on closed Hermitian manifolds.
    We also refer the reader to \cite{GTW15,Tosatti2017Weinkove,Tosatti2013Weinkove,Zhangdk2015}
    for the second order estimate \eqref{good1} of complex Monge-Amp\`ere equation for $(n-1)$-PSH functions and complex $k$-Hessian equations   on closed complex manifolds.

%Inspired by the  Liouville type theorem of Sz\'ekelyhidi \cite{Gabor},
In this paper we  derive the gradient estimate by a blow-up argument.
%and then establish the  $C^{2,\alpha}$-estimate for Dirichlet problem of fully non-linear elliptic equations on  compact Hermitian manifolds with real analytic Levi flat boundary.
To achieve this,  a specific problem that we have in mind is to establish a quantitative
version of second order boundary estimates,
 which claims  that the  complex Hessian on the boundary can
 be dominated by a quadratic term of   the boundedness  for the gradient of the unknown solution.
%%In this paper we achieve this goal.
In the following theorem, we 
%show that the real analytic Levi flat boundary enables us to 
derive such quantitative boundary estimates.

 %Combining it with Sz\'{e}kelyhidi's Liouville type theorem in \cite{Gabor}, we then establish the gradient estimate by a blow-up  argument.
%We then establish the gradient estimate by combining it with Sz\'{e}kelyhidi's Liouville type theorem in \cite{Gabor}.
   \begin{theorem}
   \label{boundaryestimatethm}
  Suppose that $(M,J,\omega)$ is a compact Hermitian manifold with real analytic Levi flat boundary.
  % $\partial M$ of class $C^2$.
  Let $\psi\in  C^{1}(\bar M)$ and $\varphi\in C^3(\partial  M)$.
Suppose, in addition to   \eqref{elliptic}-\eqref{addistruc},
 that there exists an admissible subsolution $\underline{u}\in C^{2}(\bar M)$ to  Dirichlet problem   \eqref{mainequ1}.  Then
 for any admissible function  $u \in C^3(M)\cap C^{2}(\bar M)$ solving the Dirichlet problem, we derive
    \begin{equation}
   \label{boundaryestimate1}
   \begin{aligned}
\sup_{\partial   M}|\partial\bar \partial u| \leq C(1+\sup_{ M}|\nabla u|^2),
\end{aligned}
\end{equation}
where $C$ is a uniform positive constant depending only on  $|\varphi|_{C^{3}(\bar M)}$, $|\underline{u}|_{C^{2}(\bar M)}$,   $ |\psi|_{C^{1}(\bar M)}$, $|\chi|_{C^{1}(\bar M)}$,
  $\sup_{\partial M}|\nabla u|$ % the boundary gradient bounds  $|\nabla w|_{C^0(\partial M)}$ of the supersolution given by \eqref{supersolution},
   and other known data under control (but  not on $(\delta_{\psi,f})^{-1}$).
% Moreover,  the constant $C$ in \eqref{boundaryestimate1} does not depend on $(\delta_{\psi,f})^{-1}$.
   \end{theorem}

  % \begin{remark}
  % \label{remark1-yuan}
  %Theorem \ref{boundaryestimatethm} holds for compact Hermitian manifolds with Levi flat boundary   of class $C^2$.
%Moreover,  we can prove that this quantitative boundary estimate holds under a weaker assumption of
  %  the existence of (\textit{admissible}) $\mathcal{C}$-subsolutions  with the same boundary data.
%\end{remark}

The quantitative boundary estimate   \eqref{boundaryestimate1} was  established
by Chen \cite{Chen}     %and  Phong-Sturm \cite{Phong2009Sturm}
 for Dirichlet problem of complex Monge-Amp\`ere equation on $X\times A$,
and further by Phong-Sturm \cite{Phong2009Sturm} for Dirichlet problem of
complex Monge-Amp\`ere equation  on compact K\"{a}hler manifolds with  %locally 
real analytic Levi flat boundary.
  See also Phong-Song-Sturm's survey \cite{PhongSongSturm-CMA}.
   As in the statement of  Lemma 7.17 in \cite{Boucksom2012}, %if  $\partial M$ is pseudoconcave then
 \eqref{boundaryestimate1}  indeed holds  for  complex  Monge-Amp\`{e}re
 equation on  compact K\"ahler manifolds without assuming the boundary to be real analytic Levi flat. To the best knowledge of the author, %except the work mentioned above, 
 only the literature mentioned above has studied this topic.
By using Theorems       \ref{boundaryestimatethm} and \ref{globalsecond}
we can  establish the desired second order estimate of the form \eqref{good1} for Dirichlet problem \eqref{mainequ1}.
 %has the form of
% $$\sup_{\bar M}|\partial\bar \partial  u|\leq C (1+\sup_{\bar M} |\nabla u|^2 ).$$
%\begin{equation}
%\label{good1}
%\begin{aligned}
%\sup_{\bar M}|\partial\bar \partial  u|\leq C (1+\sup_{\bar M} |\nabla u|^2 ).
%\end{aligned}
%\end{equation}
%The gradient estimate for the Dirichlet problem can be derived by a blow-up argument.
 Combining it with Sz\'{e}kelyhidi's \cite{Gabor} Liouville type theorem,
 we can prove  the gradient  bound  via a blow-up  argument, and  then
 obtain  a uniform bound of the complex Hessian
 \begin{equation}
\label{goodgood2}
\begin{aligned}
 \sup_{\bar M}|\partial\overline{\partial} u|\leq C,
 \end{aligned}
\end{equation}
where $C$ is a uniform positive constant which does not depend on the
 solution $u$ and its derivatives.
%We shall remark here that if $\Gamma=\Gamma_n$ then one can apply Proposition 5.1  of Collins-Jacob-Yau \cite{CollinsJacobYau},  in place of Sz\'ekelyhidi's Liouville type theorem,  to derive  gradient estimate.
  With the uniform estimate for complex Hessian at hand, conditions \eqref{elliptic}-\eqref{nondegenerate}
 guarantee   equation \eqref{mainequ1} to be  uniformly elliptic and concave  for \textit{admissible} solutions,
 then the Evans-Krylov theorem \cite{Evans82,Krylov82,Krylov83},
adapted to  complex setting (cf. \cite{TWWYEvansKrylov2015}), yields the $C^{2,\alpha}$ estimate.
We can also follow the line of the proof in \cite{Guan2010Li} to derive the estimates for the real Hessian so that one can directly apply Evans-Krylov theorem.
 The higher order regularity  follows from the standard Schauder  theory.

Finally, we  remark that
 it is pretty hard  to prove gradient bound directly,
  as  B{\l}ocki \cite{Blocki09gradient}, 
  % Guan \cite{Guanp2008Gradient},
  Hanani \cite{Hanani1996}  and Guan-Li \cite{Guan2010Li} did for  complex Monge-Amp\`ere equation,
   and  Zhang \cite{Zhangxw2010} did for  $\omega$-plurisubharmonic solutions of complex $k$-Hessian equations.
% and as Guan-Sun \cite{Guan2015Sun} did for complex inverse $\sigma_k$ equations including the $J$-equation.
%    P.-F. Guan \cite{Guanp2008Gradient} and did for the complex Monge-Amp\`ere equation on compact K\"ahler manifolds.
%   There are a few of cases which are known.
 %   Also, there are  several works  on the direct gradient estimate for Monge-Amp\`ere type equations on complex manifolds  \cite{Guan2010Li,Guan2015Sun,Hanani1996,Hou2009IMRN,Zhangxw2010}.
% The former works
%A straightforward extension of those works was obtained in %\cite{yuan17,yuan2019II},
 %\cite{yuan17},
 % where the author directly derived the gradient estimate of $\chi$-plurisubharmonic solutions  with $\mathfrak{g}[u] \geq 0$ for  fully non-linear elliptic  equations  %of $\Gamma=\Gamma_n$
 %on compact Hermitian manifolds (possibly with boundary), assuming the existence of $\mathcal{C}$-subsolution for a rescaled equation  $f(\frac{1}{3}\lambda(\mathfrak{g}[u]))=\psi.$ Therefore, this gives a unified and straightforward  approach to the direct gradient estimate for $\chi$-plurisubharmonic solutions of equations satisfying
 % $\lim_{R\rightarrow +\infty} f(\lambda_1,\cdots, \lambda_n+R)=+\infty \mbox{ for each }\lambda\in \Gamma.$
 %the following   condition
%\begin{equation} \label{unbound}
% \lim_{R\rightarrow +\infty} f(\lambda_1,\cdots, \lambda_n+R)=+\infty \mbox{ for each }\lambda\in \Gamma.  \nonumber
 %\end{equation}
%We shall  mention that
Also, the direct proof of gradient estimate of complex inverse
$\sigma_k$ equation was obtained in \cite{Guan2015Sun}. % including the $J$ equation.
%to \cite{yuan2018CJM} for a class of fully non-linear elliptic equations involving gradient terms including $J$-equation.
% For   ccomplex inverse $\sigma_k$ equations including the $J$-equation the gradient estimate is obtained by  Guan-Sun \cite{Guan2015Sun}.
%The result   includes  the direct gradient  estimates
%for the $\omega$-plurisubharmonic solutions of complex Monge-Amp\`ere equation and complex Hessian equation
% in  \cite{Blocki09gradient,Guanp2008Gradient,Guan2010Li,Hanani1996,Zhangxw2010} as a consequence.
%which includes complex Monge-Amp\`ere equation \cite{Blocki09gradient,Guanp2008Gradient,Guan2010Li,Hanani1996}
%and complex Hessian equations with $\omega$-plurisubharmonic solutions \cite{Zhangxw2010}.  %geometric solutions \cite{Zhangxw2010}.
%Moreover,  under a weaker assumption that the background manifold is a  compact K\"ahler manifold (possibly with boundary) with
% non-negative orthogonal bisectional curvature,
%a weaker condition of non-negative orthogonal bisectional curvature,
% a weaker restriction on non-negative orthogonal bisectional curvature,
%the author \cite{yuan17}  also obtain  the gradient estimate for general fully non-linear elliptic equations of form  \eqref{mainequ1}.
%and then extended extensively    Hou's work \cite{Hou2009IMRN}.
%for complex $k$-Hessian equations on compact K\"ahler manifolds with non-negative bisectional
%curvature due to Hou \cite{Hou2009IMRN}.
%This method also gives the direct gradient estimate for the geometric solutions of complex $k$-Hessian equations \cite{Zhangxw2010}.

\subsection{Sketch proof of Theorem \ref{boundaryestimatethm}}
\label{sketchproof}

It is  interesting  but challenging  to derive such
quantitative boundary estimates for Dirichlet problem  of
  fully non-linear elliptic equations  \eqref{mainequ1}.
%It is the first time to derive such a quantitative boundary estimate  for Dirichlet problem of
%general fully non-linear elliptic equations of the form \eqref{mainequ1}.
We propose  new insights and ideas to overcome the difficulties in this paper.
The proof of the quantitative boundary estimate is based on two ingredients:
%\begin{itemize}
%\begin{enumerate}
%\item

We set up Lemma  \ref{refinement2}
  in an attempt to bound %the boundedness for 
  double normal derivative of the solution on the boundary.
This lemma states that for a certain Hermitian matrix, if there is a diagonal element, $\mathrm{{\bf a}}$ say,  satisfying a quadratic
 growth condition, %\eqref{guanjian1-yuan},
  then the eigenvalues  concentrate near   the corresponding  diagonal elements.
 Lemma  \ref{refinement2} 
 %is in fact a  quantitative  version of Lemma 1.2 in \cite{CNS3}. and
  allows us to prove that, 
   %when the boundary is Levi flat  
   under the assumptions of Theorem \ref{boundaryestimatethm}, 
  the boundary estimates
for double normal derivative can be  dominated by the quadratic term of  the boundedness
for  tangential-normal derivatives of the \textit{admissible} solutions on the boundary (see Proposition \ref{proposition-quar-yuan1}).
 %The tangential directions of tangential-normal in this article  always lie in $T_{\partial M}\cap JT_{\partial M}$,
 %which is slightly different from the works  on complex Monge-Amp\`ere equation in literature \cite{Chen,Phong2009Sturm}   %where $\partial M$ is real analytic Levi flat.
% holomorphically flat.
% We shall remark here that the notion of Levi flat is broader than
% fairly different from
 %that of holomorphically flat
 %(cf. \cite{Lebl-Fern}).
%\item
%The Levi flatness of the boundary allows one to foliate  the boundary by complex
 %analytic hypersurfaces which are all complex local submanifolds whose complex tangent spaces are precisely
%$T_{\partial M}\cap JT_{\partial M}$.
%Please refer to  \cite{BFfoliation}  for    Barrett-Fornaess's
 %improvement  on the regularity of the induced foliation (see Lemma \ref{foliation1}).
 %Based on Cartan's theorem  %the real analytic property
 On the other hand, 
%when the boundary is real analytic Levi flat, 
 %we can  %apply % choose % the foliation  we   can choose 
%the aforementioned local holomorphic coordinate system 
 %\eqref{holomorphic-coordinate-flat} near the real analytic Levi flat  boundary to  %we can
%and use it 
we prove in Proposition \ref{Tangential-Normal derivatives}  that the
boundedness of boundary estimates for tangential-normal derivatives
depends linearly on  the 
supremum of gradient term. % (see Proposition \ref{Tangential-Normal derivatives}).
%\end{enumerate}
%\end{itemize}
Theorem \ref{boundaryestimatethm} immediately follows from these two steps.  
% Proposition \ref{Tangential-Normal derivatives} and Lemma \ref{refinement2}.

\subsection{Some phenomena on regularity assumptions}
It should be stressed that %there are several  notable and new features which seem to be noteworthy.
 there are several %new and 
 notable phenomena
on regularity assumptions on the boundary and boundary data, which are rather different
from the results and counterexamples for real Monge-Amp\`ere equations on domains in Euclidean spaces
due to Caffarelli-Nirenberg-Spruck \cite{CNS-deg} and  Wang \cite{WangXujia1996}.
%Let's explain it as in the following.

Following is the explanation.
The tangential operator on the boundary and the distance function $\sigma$ to the boundary   play   important roles in constructing the local barrier functions,
as shown in  the proof of boundary estimates  in  \cite{Caffarelli1985The,CNS1,CNS3,Hoffman1992Boundary,Guan1993Boundary,Guan1998The}.
In our case, 
%the real analytic Levi flat boundary allows us to 
we apply local coordinate \eqref{holomorphic-coordinate-flat} to 
construct a tangential operator \eqref{tangential-operator12321-meng} 
%The tangential operators on the boundary thus can be chosen as $D=\pm \frac{\partial}{\partial x_\alpha}, \pm \frac{\partial}{\partial y_\alpha}$.
%Based on  them, % and the operator $D$,
%we construct in \eqref{ggg} 
and a local barrier function $\widetilde{\Psi}$ in \eqref{ggg}.  In the proof of Proposition \ref{Tangential-Normal derivatives} we
  only use  $\sigma$, $\rho$,  $\underline{u}$ and their derivatives up to  second order, $\varphi$ and  its derivatives up to third order.
 % In addition, the constants  in  \eqref{boundaryestimate1233}, \eqref{boundaryestimate1} and \eqref{goodgood2} %and \eqref{bdr-estimate-Riemann1}
  % depend on $|\varphi|_{C^{3}(\bar M)}$.
%In order to assure  $\varphi\in C^3(\partial M)$ and the computations  of $\mathcal{L}(\widetilde{\Psi})$ (near the boundary)
%in the proof  make sense, the regularity on the boundary is  assumed to be of class $C^3$.
   It is worth stressing that 
   %when  $M:=X\times S$ is a product % (or $M$ has such a structure near the boundary)
%  (where $X$ is a closed complex manifold,  $S$ is a compact Riemann  surface with $C^2$ boundary)  and 
when the boundary data is a constant, the constant in \eqref{tangential-normal} of Proposition \ref{Tangential-Normal derivatives} depends only on
%$|\psi|_{C^{1}(\bar M)}$, $|\underline{u}|_{C^{2}(\bar M)}$,
$\partial M$ 
up to second derivatives
and other known data (see Remark \ref{remark4.4}).
%one can check that  the proof of Proposition \ref{Tangential-Normal derivatives}  also works.
%(In this case the product structure only turns out to matter near the boundary).

 %It is natural to ask if one can relax the regularity assumption on the boundary, when the boundary  is Levi flat   or totally geodesic.
%We thus see that if  $M=X\times S$ or $M$ has such a product
% structure near the boundary,
%  for the Dirichlet problem with homogeneous boundary data
%then the regularity assumption on $\partial M$   can be  further weakened to be $C^{2,\beta}$-regularity ($0<\beta\leq 1$). % which is
%% The $C^{2,\beta}$ regularity assumption on   boundary  impossible even for  the real Monge-Amp\`ere equation on certain bounded domains $\Omega\subset \mathbb{R}^2$,
%%as shown by Wang \cite{WangXujia-1996}  %for the real Monge-Amp\`ere equation on the bounded domain
%%$\Omega\subset \mathbb{R}^n$,
%% the optimal regularity  assumptions on
%%the boundary and boundary data  are both $C^3$ for such real Monge-Amp\`ere equations.
%%%It is well known that the boundary $\partial \Omega$ of any
%%% $C^2$  bounded domain $\Omega\in \mathbb{R}^n$ has a strictly convex point.
%We shall remark here that the $C^{2,\alpha}$ boundary regularity follows from a result of Silvestre-Sirakov \cite{Silvestre2014Sirakov}.
%%(I would like to  thank  Connor Mooney for informing    Silvestre-Sirakov's results).

As a result, the regularity  assumptions in  Theorem \ref{existence} % and \ref{existencede} %and   \ref{existenceRiemann}
can be further weakened when $M=X\times S$ is a product of a closed complex manifold $X$ with a compact Riemann surface $S$ with boundary.  More precisely,
 if $\psi\in C^{2}(\bar M)$ as well as $\partial M\in C^3$, $\varphi\in C^3$,
then the solution $u$ % in  Theorem \ref{existence}
 lies in  $ C^{2,\alpha}(\bar M)$ with   $0<\alpha<1$.
%A wonderful result of this paper is that when the boundary $\partial M$ is   Levi flat  or totally geodesic
%(in the sense that  one can pick  local holomorphic coordinates $(z_1,\cdots, z_n)$ such that $\partial M$ is
% locally of the form $\mathfrak{Re}(z_n)=0$), %(respectively, totally geodesic),
%the $C^3$ regularity assumptions in the corresponding theorems can be further
%weakened to $\partial M\in C^{2,1}$ and  $\varphi\in C^3$.
 At least on some convex domains $\Omega\subset \mathbb{R}^2$,
such $C^3$ regularity assumptions on boundary and  boundary data are optimal for Dirichlet problem of
 nondegenerate real Monge-Amp\`ere equations
 as shown by   Wang  \cite{WangXujia1996}, in which
he constructed counterexamples to show that
 if either $\partial \Omega$ or $\varphi$ is only $C^{2,1}$ smooth, then the solution $u$
 to real Monge-Amp\`ere equation on $\Omega\in \mathbb{R}^2$ may fail to be $C^2$ smooth near the boundary.
  For the Dirichlet problem with homogeneous boundary data on %Hermitian manifolds of the type %of
  products  
   $X\times S$,
    the regularity assumption on the boundary may be further relaxed  to $C^{2,\beta}$ ($0<\beta<1$)
   in this special case.  We shall remark here that the $C^{2,\alpha}$ boundary regularity may follow  from a result of Silvestre-Sirakov \cite{Silvestre2014Sirakov}.

 \begin{theorem}
\label{dege-thm-c2alpha}
%In addition to $\varphi=0$,  %$\psi\in C^{2}(M )\cap C^{1,1}(\bar M)$, 
%$\psi\in C^{2}(\bar M) $, we  
Assume $M=X\times S$, 
 %(or $M$ has such a product structure near the boundary),
where $X$ is a closed complex manifold and $S$ is a compact Riemann surface with
$C^{2,\beta}$  boundary  ($0<\beta<1$). Suppose, in addition to  \eqref{elliptic}-\eqref{addistruc}, $\varphi=0$ and %$\psi\in C^{2}(M )\cap C^{1,1}(\bar M)$, 
$\psi\in C^{2}(\bar M)$, that Dirichlet
problem  \eqref{mainequ1} has a $C^2$-smooth admissible subsolution. Then the Dirichlet problem
  supposes a unique $C^{2,\alpha}$ \textit{admissible} solution, where $0<\alpha<1$ depends on $\beta$ and other known data.

\end{theorem}

For the Dirichlet problem of degenerate fully non-linear elliptic equations on $M=X\times S$,
the regularity assumptions in Theorem \ref{existencede} %and   \ref{existencedeRiemann} 
can also be further weakened to 
$\varphi\in C^{2,1}$ and $\partial M\in C^{2,1}$.
% when $M=X\times S$.
Furthermore, as in Theorem \ref{dege-thm-c2alpha}  we have the corresponding
 result for degenerate equations when $M=X\times S$.
 % (or $M$ has such a product structure near the boundary).
 %The product structure only turns out to matter near the boundary.
Such regularity assumptions on boundary and   boundary data
are impossible for homogeneous real Monge-Amp\`ere equation on certain
bounded domains $\Omega\subset \mathbb{R}^n$ as shown by %some counterexamples in 
\cite{CNS-deg}, in which
%Caffarelli-Nirenberg-Spruck  presented some counterexamples to show that
 the $C^{3,1}$-regularity assumptions on   boundary and boundary data
are optimal for the optimal $C^{1,1}$ global regularity of the weak solution to homogeneous real Monge-Amp\`ere equation 
on $\Omega$.

 Notice that the boundary of a
 $C^2$  bounded domain $\Omega\in \mathbb{R}^n$ has  at least one strictly convex point at which every principal curvature of $\partial\Omega$ is positive, to be compared with the Levi flatness of $M=X\times S$.
% the global regularity of the solutions for   fully non-linear elliptic equations
%This is part of a series of papers which are  devoted to investigating homogeneous complex Monge-Amp\`ere equation
%and degenerate fully non-linear elliptic equations in complex variable setting.
%We shall note that this  is a new version of the previous manuscript which was completed in 2017.
%In this new version  we append   Section \ref{HCMA-result}  %in which we  apply Li's constant rank theorem in \cite{LiQun2009}
% to study the  connection between  homogeneous complex Monge-Amp\`ere equation and  non-degenerate complex Hessian equations on
 %certain K\"ahler manifolds. % with positive orthogonal bisectional curvature.
%we apply Q. Li's constant rank theorem in \cite{LiQun2009} to connect the homogeneous complex Monge-Amp\`ere equation with  non-degenerate complex Hessian equations on  certain K\"ahler manifolds.
%The new features on regularity assumption mentioned above and counterexamples in 
%\cite{CNS-deg}  
%\cite{CNS-deg,WangXujia1996} 
These new features show that the assumption on the boundary
%of $\partial M$ in Theorem \ref{existencede} and Theorem  \ref{existencedeRiemann}, respectively,
in the corresponding theorem %(i.e. $\partial M$  is  locally flat or \textit{concave}) 
is essential for improving the regularity assumptions
on the boundary and the boundary data.

%\noindent{\bf Organization}
 \subsection*{Organization}
 %\vspace{1.5mm}
The rest of this paper is organized as follows.
 In Section \ref{Preliminaries} we outline the notation and  some useful results.
In Section \ref{refinementofCNS}
  we derive a quantitative lemma which is the crucial ingredient in the proof of quantitative boundary estimates.
%we will give more precise quantitative version of of Caffarelli-Nirenberg-Spruck's lemma, which is
%a crucial ingredient  in the proof of sharp boundary estimates.
In Section \ref{preciseboundaryestimates} we  establish quantitative boundary estimates.
%In Section \ref{GlobalSecondEstimates11} we obtain the global second order estimates.
In Section \ref{solvingequation} we solve the Dirichlet problem by using the method of continuity and approximation.
%In Section \ref{HCMA-result} we apply the estimate  \eqref{goodgood2} and  Li's  \cite{LiQun2009} constant rank theorem to study homogeneous complex Monge-Amp\`ere equation on the corresponding  K\"ahler manifolds.
In Section \ref{RiemannDirichlet} we briefly investigate the solvability and  regularity of Dirichlet problem for
fully non-linear elliptic equations on Riemannian manifolds with concave boundary. % and investigate the regularity of  weak solutions of degenerate equations.
%Finally,   % a   characterization of level sets of $f$ which satisfies \eqref{elliptic}, \eqref{concave} and \eqref{addistruc} is given in Appendix  \ref{Appendix}.
  In Appendix \ref{Appendix} we finally present
a characterization of the level sets of $f$   satisfying  \eqref{elliptic}, \eqref{concave} and  \eqref{addistruc},
which plays an important role in proving the boundary estimates for `tangential-normal' derivatives.

%\noindent{\bf Acknowledgements}
 \subsection*{Acknowledgements}
%I am very grateful to %my supervisors 
%Professors Bo Guan and Chunhui Qiu for their constant support and warm encouragement over the past several years.
% his   support and for introducing  me to the area of fully non-linear  equations,
%to Professor Chunhui Qiu for the support and encouragement.
%I would also wish to express my gratitude to Professors Chunhui Qiu and   Xi Zhang for their   support  and warm encouragements.
The work was done while I was visiting University of Science and Technology of China in semester  2016/2017.
I wish to express my gratitude to  Professor Xi Zhang for his kindly support and thank the Department and the
University for their hospitality.
%Thanks also go to Professor Xi-Nan Ma  for his encouragement. 
%I would also wish to thank Professors Bo Guan and Chunhui Qiu and for their  support and warm encouragements.

%This is the first paper of series researches to understand the Dirichlet problem for fully nonlinear elliptic equations on complex manifolds, which include arXiv:2001.09238, arXiv:2106.14837 and (Yuan, Pure Appl. Math. Q. 16: 1585-1617, 2020).
%I  would like to thank the referee for the valuable comments and suggestions. 

%I  dedicate this article  to the memory of my grandfather,   and to the memory of my mother. Also, special thanks  go to my grandmother, my father and the rest of my family for their unfailing support.

\section{Preliminaries}
\label{Preliminaries}

In this section, we  outline and set up some useful results.
For convenience we first denote  $$\vec{1}=(1,\cdots,1), \mbox{  }
\lambda[v]=\lambda(\mathfrak{g}[v]),  \mbox{  } \lambda=\lambda[u], \mbox{  } \underline{\lambda}=\lambda[\underline{u}].$$

 %$\lambda[v]=\lambda(\mathfrak{g}[v])$, $\lambda=\lambda[u]$, $\underline{\lambda}=\lambda[\underline{u}]$ for convenience.

%To study the degenerate  equation with $\delta_{\psi,f}= 0$,
%we  use  Lemma \ref{asymptoticcone1} in the appendix  to  obtain an explicitly lower bound of $\sum f_i(\lambda)$.

Based on a characterization of level sets of $f$ satisfying \eqref{elliptic}, \eqref{concave} and \eqref{addistruc},
we  give a new proof of   Part (b) of Lemma 9 in \cite{Gabor}.
%Comparing with Sz\'ekelyhidi's original proof,
%our proof gives an explicitly lower bound of $\sum_{i=1}^n f_i(\lambda)$.
%and then obtain an explicitly lower bound of $\sum_{i=1}^n f_i(\lambda)$.
%Moreover, we can apply it to study the degenerate  equation with $\delta_{\psi,f}= 0$.
We  state it as follows.
\begin{lemma}
\label{lxf4}
Suppose   \eqref{elliptic}, \eqref{concave} and  \eqref{addistruc} hold.
%Let $\sigma\in (\sup_{\partial\Gamma}f, \sup_{\Gamma}f)$,
Then for $\sigma\in (\sup_{\partial\Gamma}f, \sup_{\Gamma}f)$ and $t>0$, we have
\begin{equation}
\label{yuan-lxf4}
\begin{aligned}
\sum_{i=1}^n f_i(\lambda) > \frac{f(t\vec{1})-\sigma}{t} \mbox{ in } \partial\Gamma^\sigma:=\{\lambda\in \Gamma: f(\lambda)=\sigma\}.  %\nonumber
\end{aligned}
\end{equation}
In particular, %if $\lambda=\lambda(\mathfrak{g}[u])\in \Gamma$ satisfies the equation $f(\lambda(\mathfrak{g}[u]))=\psi \mbox{ in }  M,$
if $u$ is a $C^2$ admissible solution of equation \eqref{mainequ1}
 then
\begin{equation}
\label{lxf3}
\begin{aligned}
\sum_{i=1}^n f_i(\lambda(\mathfrak{g}[u]))\geq \kappa>0.
\end{aligned}
\end{equation}
where   $\kappa=\frac{f((1+c_0)\vec{1})-\sup_M \psi}{1+c_0}$, which is  independent of
 $\delta_{\psi,f}$, and $c_0$ is the positive constant satisfying $f(c_0\vec{1})=\sup_{M}\psi$.
\end{lemma}
%We know that the   lower bound of $\sum_{i=1}^ n f_i(\lambda)$    was proved
%by Sz\'ekelyhidi (see Part (b) of Lemma 9 in \cite{Gabor}).

    % prove that $\sum_{i=1}^n f_i(\lambda)$ has a  positive lower bound.

%We present our new proof as in the following.
 \begin{proof}
% Let $c_0$ be the positive constant satisfying $f(c_0\vec{1})=\sup_{M}\psi$.
In Appendix \ref{Appendix} we obtain a characterization of level sets of $f$ when it satisfies \eqref{elliptic}, \eqref{concave} and \eqref{addistruc}.
The proof of Lemma \ref{lxf4} is based on this characterization.
%in Appendix \ref{Appendix}.
Fix $t>0$.
By \eqref{concave}   and Lemma \ref{asymptoticcone1},   one has
\begin{equation}
\begin{aligned}
t\sum f_i(\lambda)\geq  \sum_{i=1}^n  f_i(\lambda)\lambda_i+f(t\vec{1})-f(\lambda)
>   f(t\vec{1})-\sigma  \mbox{ in }  \partial\Gamma^\sigma.  \nonumber
\end{aligned}
\end{equation}
 Furthermore, \eqref{lxf3} holds for $\kappa=\frac{f((1+c_0)\vec{1})-\sup_M \psi}{1+c_0}$ by setting $t= 1+c_0$.
 \end{proof}

 An important ingredient in the proof of  \textit{a priori} estimates for fully non-linear elliptic equations is the following lemma.
 %Lemma \ref{guan2014} or Lemma \ref{gabor'lemma}.
 %These two lemmas also play similar roles in proving Theorem \ref{boundaryestimatethm}.
\begin{lemma}
[\cite{GSS14}, Lemma 2.2]
\label{guan2014}
% Suppose that $f$ satisfies \eqref{elliptic} and \eqref{concave}. For $0<\beta<1$ and $\sigma\in (\sup_{\partial \Gamma}f, \sup_\Gamma f)$
% there is a constant $\varepsilon>0$ such that, when $|\nu_{\mu}-\nu_{\lambda}|\geq \beta$, one has
%Let $f$ be a symmetric function on the open symmetric and convex cone $\Gamma$.
Suppose that $f$ satisfies \eqref{elliptic} and \eqref{concave}.  % and $\Gamma_{n}\subseteq \Gamma\subset \Gamma_{1}$.
 Let $K$ be a compact subset of $\Gamma$ and $\beta>0$. There is a constant $\varepsilon>0$ such that,
 for any $\mu\in K$ and $\lambda\in \Gamma$, when $|\nu_{\mu}-\nu_{\lambda}|\geq \beta$,
\begin{equation}
\label{2nd}
\begin{aligned}
\sum f_{i}(\lambda)(\mu_{i}-\lambda_{i})\geq f(\mu)-f(\lambda)+\varepsilon (1+\sum f_{i}(\lambda)),
\end{aligned}
\end{equation}
where $\nu_{\lambda}=Df(\lambda)/|Df(\lambda)|$ denotes the unit normal vector to the level surface of
 $f$ through $\lambda$.

%\begin{equation}\label{2nd} \begin{aligned}
%\sum_{i=1}^n f_{i}(\lambda)(\mu_{i}-\lambda_{i})\geq f(\mu)-f(\lambda)+\varepsilon  (1+\sum_{i=1}^n f_{i}(\lambda)) \mbox{ in } \partial\Gamma^\sigma,
%\end{aligned}\end{equation}
%where $\nu_{\lambda}=Df(\lambda)/|Df(\lambda)|$ denotes the unit normal vector to the level surface of $f$ passing through $\lambda$.
\end{lemma}
For the given \textit{admissible} subsolution $\underline{u}$,
$\underline{\lambda}$ falls in a compact subset of $\Gamma$. We take
\begin{equation}
\label{beta}
\begin{aligned}
\beta:= \frac{1}{2}\min_{\bar M} \mathrm{dist}(\nu_{\underline{\lambda}}, \partial \Gamma_{n})>0.
\end{aligned}
\end{equation}
%To apply the above lemma, if $|\nu_{\lambda}-\nu_{\underline{\lambda}}|\geq \beta$  then
It follows from  Lemma \ref{guan2014} and Lemma 6.2 in \cite{CNS3}
 that when $|\nu_{\lambda }-\nu_{\underline{\lambda}}|\geq \beta$  we have
\begin{equation}
\begin{aligned}
\mathcal{L}(\underline{u}-u)\geq \varepsilon(1 + \sum_{i=1}^n f_i(\lambda)). \nonumber
\end{aligned}
\end{equation}
While   $|\nu_{\lambda}-\nu_{\underline{\lambda}}|< \beta$  ensures
$\nu_{\lambda}-\beta \vec{1} \in \Gamma_{n}$ and
\begin{equation}
\begin{aligned}
f_{i} (\lambda)\geq  \frac{\beta }{\sqrt{n}} \sum_{j=1}^n f_{j}(\lambda) \mbox{ for each } i=1,\cdots, n. \nonumber
\end{aligned}
\end{equation}
From the original proof of Lemma 2.2 in \cite{GSS14},
we know that the constant $\varepsilon$ in this lemma depends only on
$\underline{\lambda}$, $\beta$ and other known data.
 Please refer to \cite{GSS14} for the detail.

In analogy with Lemma \ref{guan2014}, Sz\'ekelyhidi  also proved the following lemma.
%which states as follows:  % (see  Proposition 5 in \cite{Gabor}):
\begin{lemma}
 [\cite{Gabor}, Proposition 5]
\label{gabor'lemma}
Let $f$ be the function satisfying \eqref{elliptic} and \eqref{concave},
and let $\sup_{\partial \Gamma}f <\sigma<\sup_\Gamma f$.
Suppose that there exists a  $\mathcal{C}$-subsolution $\underline{u}\in C^{2}(\bar M)$.
Then there exist two uniform positive constants $R_0$ and $ \varepsilon_1$,
such that if $\lambda\in \partial\Gamma^\sigma$ and
 $|\lambda|\geq R_0$, then   either
\begin{equation}
\label{haha1}
\begin{aligned}
\sum_{i=1}^n f_i(\lambda)(\underline{\lambda}_i - \lambda_i)
\geq  \varepsilon_1 \sum_{j=1}^n f_j (\lambda),
\end{aligned}
\end{equation}
or $f_i(\lambda)\geq  \varepsilon_1 \sum_{j=1}^n f_j (\lambda) \mbox{ for each } i=1,\cdots, n.$
%\begin{equation}
%\begin{aligned}
%f_i(\lambda)\geq  \varepsilon_1 \sum_{j=1}^n f_j (\lambda) \mbox{ for each } i=1,\cdots, n.
%\end{aligned}
%\end{equation}
\end{lemma}

It shall be noted that any \textit{admissible} subsolution satisfying \eqref{existenceofsubsolution}
  is clearly the $\mathcal{C}$-subsolution  introduced by Sz\'ekelyhidi \cite{Gabor},  so we can
  apply either  Lemma \ref{guan2014} or Lemma \ref{gabor'lemma}  to derive the estimates established in this paper.

 \vspace{1mm}
Throughout this paper we use derivatives with respect to Chern connection $\nabla$
 associated with $\omega$, and in local coordinates $z = (z_1,\cdots,z_n)$
 use notations such as
\begin{equation}
\begin{aligned}
v_i=\nabla_{\frac{\partial}{\partial z_{i}}} v, \mbox{  }
 v_{ij}=\nabla_{\frac{\partial}{\partial z_{j}}}\nabla_{\frac{\partial}{\partial z_{i}}} v,
 \mbox{  }
v_{i\bar j}=\nabla_{\frac{\partial}{\partial \bar z_{j}}}\nabla_{\frac{\partial}{\partial z_{i}}} v, \cdots.  \nonumber
\end{aligned}
\end{equation}

We denote
$\mathfrak{g}=\mathfrak{g}[u]$ and $\mathfrak{\underline{g}}=\mathfrak{g}[\underline{u}]$
for the solution $u$ and the subsolution $\underline{u}$ respectively.
Given a Hermitian matrix $A=\{a_{i\bar j}\}$, we write $F^{i\bar j}(A) = \frac{\partial F}{\partial a_{i\bar j}} (A).$

%\begin{equation}
%\begin{aligned}
%F^{i\bar j}(A) = \frac{\partial F}{\partial a_{i\bar j}} (A), \mbox{  }
%F^{i\bar j, k\bar l}(A) = \frac{\partial^{2} F}{\partial a_{i\bar j}\partial a_{k\bar l}} (A).  \nonumber
%\end{aligned}
%\end{equation}

\section{Quantitative Lemmas}
\label{refinementofCNS}

 In this section we set up a lemma which states that
 if the parameter $\mathrm{{\bf a}}$ satisfies a quadratic
 growth condition %\eqref{guanjian1-yuan}
 then  the eigenvalues concentrate %  at the scale of $\epsilon$,  
 near   the corresponding diagonal elements.
 Namely,
 \begin{lemma}
\label{refinement2}
Let $A$ be an $n\times n$ Hermitian matrix
\begin{equation}\label{matrix3}\left(\begin{matrix}
d_1&&  &&a_{1}\\ &d_2&& &a_2\\&&\ddots&&\vdots \\ && &  d_{n-1}& a_{n-1}\\
\bar a_1&\bar a_2&\cdots& \bar a_{n-1}& \mathrm{{\bf a}} \nonumber
\end{matrix}\right)\end{equation}
with $d_1,\cdots, d_{n-1}, a_1,\cdots, a_{n-1}$ fixed, and with $\mathrm{{\bf a}}$ variable.
Denote $\lambda=(\lambda_1,\cdots, \lambda_n)$ by   the eigenvalues of $A$.
%With the same notation in Lemma \ref{refinement3}.
Let $\epsilon>0$ be a fixed constant.
Suppose that  the parameter $\mathrm{{\bf a}}$ in $A$ satisfies  the quadratic
 growth condition  
  \begin{equation}
 \begin{aligned}
\label{guanjian1-yuan}
\mathrm{{\bf a}}\geq \frac{2n-3}{\epsilon}\sum_{i=1}^{n-1}|a_i|^2 +(n-1)\sum_{i=1}^{n-1} |d_i|+ \frac{(n-2)\epsilon}{2n-3},
\end{aligned}
\end{equation}
where  $\epsilon$ is a positive constant.
% \begin{equation}
%a \geq \frac{2}{\epsilon}\sum_{i=1}^{n-1}|a_i|^2 +(n-1)\sum_{i=1}^{n-1} |d_i|+n \epsilon. \nonumber
% \end{equation}
 Then the eigenvalues 
 %(possibly  with a proper order)
 % (with a proper permutation)
   behavior like
\begin{equation}
\begin{aligned}
\,& |d_{\alpha}-\lambda_{\alpha}|
<   \epsilon, \mbox{ }\forall 1\leq \alpha\leq n-1,\\
\,&0\leq \lambda_{n}-\mathrm{{\bf a}}
< (n-1)\epsilon. \nonumber
\end{aligned}
\end{equation}
\end{lemma}
 %This allow us to follow  the track  of the behavior of  the eigenvalues $\lambda_i$  as $|a|$ tends to infinity.
 %This lemma is a crucial ingredient  for our method in this paper,
 This lemma is in fact a  quantitative version of the following lemma. %  (see  Lemma 1.2 in \cite{CNS3}).
%In this section we prove a quantitative lemma  which is a crucial ingredient  for our method in this paper.
%The  lemma is a  quantitative versions of  Lemma \ref{lemmaCNS3} below.
 %which is proved by Caffarelli-Nirenberg-Spruck  \cite{CNS3}  (see also Lemma 1.2 in \cite{CNS3}).
  %\begin{equation}
%\label{unbound}
%\begin{aligned}
%\lim_{R\rightarrow +\infty} f(\lambda_1,\cdots, \lambda_n+R)=+\infty \mbox{ for each }\lambda\in \Gamma.
%\end{aligned}
%\end{equation}
 %And this condition above was further removed by Trudinger \cite{Trudinger95}.
\begin{lemma}
[\cite{CNS3}, Lemma 1.2]
\label{lemmaCNS3}
Consider the $n\times n$ symmetric matrix
\begin{equation}
\label{matrix1}
A=\left(
\begin{matrix}
d_1&&  &&a_{1}\\
&d_2&& &a_2\\
&&\ddots&&\vdots \\
&& &  d_{n-1}& a_{n-1}\\
a_1&a_2&\cdots& a_{n-1}& \mathrm{{\bf a}}              \nonumber
\end{matrix}
\right)
\end{equation}
with $d_1,\cdots, d_{n-1}$ fixed, $| \mathrm{{\bf a}} |$ tends to infinity and
\begin{equation}
|a_i|\leq C, i=1,\cdots, n. \nonumber
\end{equation}
Then the eigenvalues $\lambda_1,\cdots, \lambda_{n}$ behave like
\begin{equation}
\label{behave1}
\begin{aligned}
\,&
\lambda_{\alpha}=d_{\alpha}+o(1),\mbox{ } 
1\leq \alpha \leq n-1,\\
\,&
\lambda_{n}=\mathrm{{\bf a}}\left(1+O\left(1/\mathrm{{\bf a}}\right)\right),      \nonumber
\end{aligned}
\end{equation}
where the $o(1)$ and $O(1/\mathrm{{\bf a}})$ are uniform--depending only on $d_{1},\cdots, d_{n-1}$ and $C$.
\end{lemma}
 In their seminal paper, Caffarelli-Nirenberg-Spruck  \cite{CNS3}  proved
  Lemma \ref{lemmaCNS3}   and
applied it to derive boundary
 estimates for double normal derivative of
  certain fully non-linear elliptic equations on $\Omega\subset \mathbb{R}^n$.
  % under certain hypotheses  including   condition \eqref{unbound}.
 Since then Lemma \ref{lemmaCNS3}  plays an important role in the boundary estimates for second derivatives
 in the study of Dirichlet problem for general fully non-linear elliptic equations (cf.  \cite{LiYY1990,Trudinger95,Guan12a} and references therein).

 %the requirement that

Our main estimate in this paper is  the quantitative boundary estimate
which states that the complex Hessian on the boundary can be
bounded by a quadratic term of   the boundedness  for the gradient of unknown solutions.
% which allows us to  derive the gradient estimate by a scaling argument.
 % a specific problem that we have in mind is to establish the
%  quantitative boundary estimate (Theorem \ref{boundaryestimatethm}).
Except complex Monge-Amp\`ere equation,
such  quantitative boundary estimates for  general fully non-linear elliptic equations
 are poorly understood % in  mathematicians' literature
 on this subject,
because Lemma \ref{lemmaCNS3}  does not figure out
how do  the eigenvalues   concentrate explicitly near the corresponding
 diagonal elements  of  the matrix $A$ when $|\mathrm{{\bf a}}|$ is sufficiently large.

We  make progress on the subject.
  Lemma \ref{refinement2} allows us to follow  the track  of the behavior of  the eigenvalues  as $|\mathrm{{\bf a}}|$ tends to infinity.
% the detailed   behavior
%of the eigenvalues $\lambda_i$   when $|\mathrm{{\bf a}}|$ tends to infinity.
%That is the reason  why Lemma \ref{lemmaCNS3}  can not figure out how the boundary estimates depend on the  gradient term.
%To be specific, the eigenvalues $(\lambda_1,\cdots,\lambda_{n-1},\lambda_n)$
%would  concentrate around $(d_1,\cdots, d_{n-1}, a)$ if $|a|=$
%  Hence  the quantitative lemmas help us making the boundary estimates clear.
 %The quantitative lemmas  can be seen as the quantitative version  of Lemma \ref{lemmaCNS3}.
%The crucial ingredient in the proof of  the  quantitative boundary estimates  are the following lemmas.
%\begin{remark}
%\label{lemma-n2}

We start with the case of $n=2$. In this case,
%  we   propose a quantitative version of Lemma \ref{lemmaCNS3}.
we prove that  if $\mathrm{{\bf a}} \geq \frac{|a_1|^2}{ \epsilon}+ d_1$
%\begin{equation}
%\label{guanjian1-n2}
%\begin{aligned}
%\mathrm{{\bf a}} \geq \frac{|a_1|^2}{ \epsilon}+ d_1
%\end{aligned}
%\end{equation}
%%$$a>\frac{|a_1|^2}{ \epsilon}+d_1-\epsilon $$
 then
$$0\leq d_1- \lambda_1=\lambda_2-\mathrm{{\bf a}} <\epsilon.$$

Let's briefly  present the discussion as follows:
For $n=2$, the eigenvalues of $A$ are
 $$\lambda_{1}=\frac{\mathrm{{\bf a}}+d_1- \sqrt{(\mathrm{{\bf a}}-d_1)^2+4|a_1|^2}}{2} 
\mbox{ and } \lambda_2=\frac{\mathrm{{\bf a}}+d_1+\sqrt{(\mathrm{{\bf a}}-d_1)^2+4|a_1|^2}}{2}.$$ 
We can assume $a_1\neq 0$; otherwise we are done.
If $\mathrm{{\bf a}} \geq \frac{|a_1|^2}{ \epsilon}+ d_1$ then one has
\begin{equation}
\begin{aligned}
0\leq d_1- \lambda_1 =\lambda_2-\mathrm{{\bf a}}
= \frac{2|a_1|^2}{\sqrt{ (\mathrm{{\bf a}}-d_1)^2+4|a_1|^2 } +(\mathrm{{\bf a}}-d_1)}
< \frac{|a_1|^2}{\mathrm{{\bf a}}-d_1 } \leq \epsilon.   \nonumber
\end{aligned}
\end{equation}
Here we use $a_1\neq 0$ to verify that the strictly inequality in the above formula holds.
We hence obtain Lemma \ref{refinement2} for $n=2$.

The following lemma enables us  to count  the eigenvalues near the diagonal elements
via a deformation argument.
It is an essential  ingredient in the proof of  Lemma \ref{refinement2}  for general $n$.
\begin{lemma}
\label{refinement}
Let $A$ be a Hermitian $n$ by  $n$  matrix
\begin{equation}
\label{matrix2}
\left(
\begin{matrix}
d_1&&  &&a_{1}\\
&d_2&& &a_2\\
&&\ddots&&\vdots \\
&& &  d_{n-1}& a_{n-1}\\
\bar a_1&\bar a_2&\cdots& \bar a_{n-1}& \mathrm{{\bf a}} \nonumber
\end{matrix}
\right)
\end{equation}
with $d_1,\cdots, d_{n-1}, a_1,\cdots, a_{n-1}$ fixed, and with $\mathrm{{\bf a}}$ variable.
Denote
$\lambda=(\lambda_1,\cdots, \lambda_n)$ by the the eigenvalues of $A$ with the order
$\lambda_1\leq \lambda_2 \leq\cdots \leq \lambda_n$.
Fix   a positive constant $\epsilon$.
Suppose that the parameter $\mathrm{{\bf a}}$ in the matrix $A$ satisfies  the following quadratic growth condition
%$$a>\frac{2}{\epsilon}\sum_{i=1}^{n-1} a_i^2+ (n-1)\sum_{i=1}^{n-1} |d_i|+ (n-2)\epsilon.$$
\begin{equation}
\label{guanjian2}
\begin{aligned}
\mathrm{{\bf a}} \geq \frac{1}{\epsilon}\sum_{i=1}^{n-1} |a_i|^2+\sum_{i=1}^{n-1}  [d_i+ (n-2) |d_i|]+ (n-2)\epsilon.
\end{aligned}
\end{equation}
Then for   any $\lambda_{\alpha}$ $(1\leq \alpha\leq n-1)$ there exists an  $d_{i_{\alpha}}$
with lower index $1\leq i_{\alpha}\leq n-1$ such that
\begin{equation}
\label{meishi}
\begin{aligned}
 |\lambda_{\alpha}-d_{i_{\alpha}}|<\epsilon,
\end{aligned}
\end{equation}
\begin{equation}
\label{mei-23-shi}
0\leq \lambda_{n}-\mathrm{{\bf a}} <(n-1)\epsilon + |\sum_{\alpha=1}^{n-1}(d_{\alpha}-d_{i_{\alpha}})|.
\end{equation}
\end{lemma}

\begin{proof}
%[Proof of Lemma \ref{refinement}]
Without loss of generality, we assume $\sum_{i=1}^{n-1} |a_i|^2>0$ and  $n\geq 3$
(otherwise we are done, since $A$ is diagonal or $n=2$).
%If $\sum_{i=1}^{n-1} |a_i|^2=0$, then $A$ is diagonal, and  we are done.
%Suppose now $\sum_{i=1}^{n-1} |a_i|^2>0$.
Note that in the assumption of the lemma the eigenvalues have
the order $\lambda_1\leq \lambda_2\leq \cdots \leq \lambda_n$.
It is  well known that, for a Hermitian matrix,
 any diagonal element is   less than or equals to   the  largest eigenvalue.
 In particular,
 \begin{equation}
 \label{largest-eigen1}
 \lambda_n \geq \mathrm{{\bf a}}.
 \end{equation}

We only need to prove   \eqref {meishi}, since  \eqref{mei-23-shi} is a consequence of  \eqref{meishi}, \eqref{largest-eigen1}  and
\begin{equation}
\label{trace}
 \sum_{i=1}^{n}\lambda_i=\mbox{tr}(A)=\sum_{\alpha=1}^{n-1} d_{\alpha}+\mathrm{{\bf a}}.
 \end{equation}

 Let's denote   $I=\{1,2,\cdots, n-1\}$. We divide the index set   $I$ into two subsets  by
$${\bf B}=\{\alpha\in I: |\lambda_{\alpha}-d_{i}|\geq \epsilon, \mbox{   }\forall i\in I\} $$
and $ {\bf G}=I\setminus {\bf B}=\{\alpha\in I: \mbox{There exists an $i\in I$ such that } |\lambda_{\alpha}-d_{i}| <\epsilon\}.$

To complete the proof we need to prove ${\bf G}=I$ or equivalently ${\bf B}=\emptyset$.
  It is easy to see that  for any $\alpha\in {\bf G}$, one has
   \begin{equation}
   \label{yuan-lemma-proof1}
   \begin{aligned}
   |\lambda_\alpha|< \sum_{i=1}^{n-1}|d_i| + \epsilon.
   \end{aligned}
   \end{equation}

   Fix $ \alpha\in {\bf B}$,  we are going to give the estimate for $\lambda_\alpha$.
%Let  $\lambda$ be an eigenvalue of $A$,  then
%Since   $\lambda_{\alpha}$ is an eigenvalue of $A$,  one has
The eigenvalue $\lambda_\alpha$ satisfies
\begin{equation}
\label{characteristicpolynomial}
\begin{aligned}
(\lambda_{\alpha} -\mathrm{{\bf a}})\prod_{i=1}^{n-1} (\lambda_{\alpha}-d_i)
= \sum_{i=1}^{n-1} (|a_{i}|^2 \prod_{j\neq i} (\lambda_{\alpha}-d_{j})).
\end{aligned}
\end{equation}
By the definition of ${\bf B}$, for  $\alpha\in {\bf B}$, one then has $|\lambda_{\alpha}-d_i|\geq \epsilon$ for any $i\in I$.
We therefore derive
\begin{equation}
\begin{aligned}
|\lambda_{\alpha}-\mathrm{{\bf a}} |\leq \sum_{i=1}^{n-1} \frac{|a_i|^2}{|\lambda_{\alpha}-d_{i}|}\leq
\frac{1}{\epsilon}\sum_{i=1}^{n-1} |a_i|^2, \mbox{ if } \alpha\in {\bf B}.
\end{aligned}
\end{equation}
%Here we use $\sum_{i=1}^{n-1} |a_i|^2>0$.
Hence,  for $\alpha\in {\bf B}$, we obtain
\begin{equation}
\label{yuan-lemma-proof2}
\begin{aligned}
 \lambda_\alpha \geq \mathrm{{\bf a}}-\frac{1}{\epsilon}\sum_{i=1}^{n-1} |a_i|^2.
\end{aligned}
\end{equation}
%Furthermore, $\lambda_n \geq \mathrm{{\bf a}}-\frac{1}{\epsilon}\sum_{i=1}^{n-1} |a_i|^2$,
%as one assumes   $\lambda_n\geq \lambda_\alpha$ in the lemma.

For a set ${\bf S}$, we denote $|{\bf S}|$ the  cardinality of ${\bf S}$.
We shall use proof by contradiction to prove  ${\bf B}=\emptyset$.
Assume ${\bf B}\neq \emptyset$.
Then $|{\bf B}|\geq 1$, and so $|{\bf G}|=n-1-|{\bf B}|\leq n-2$. % and we will prove that it is a contradiction.

%In the case of ${\bf G}\neq \emptyset$, 
We compute the trace of the matrix $A$ as follows:
\begin{equation}
\begin{aligned}
\mbox{tr}(A)=\,&
\lambda_n+
\sum_{\alpha\in {\bf B}}\lambda_{\alpha} + \sum_{\alpha\in  {\bf G}}\lambda_{\alpha}\\
> \,&
\lambda_n+
|{\bf B}| (\mathrm{{\bf a}}-\frac{1}{\epsilon}\sum_{i=1}^{n-1} |a_i|^2 )-|{\bf G}| (\sum_{i=1}^{n-1}|d_i|+\epsilon ) \\
\geq \,&
 2\mathrm{{\bf a}}-\frac{1}{\epsilon}\sum_{i=1}^{n-1} |a_i|^2 -(n-2) (\sum_{i=1}^{n-1}|d_i|+\epsilon )
\\
%\geq\,& \sum_{i=1}^{n-1} |d_i|+\mathrm{{\bf a}} \\
\geq \,& \sum_{i=1}^{n-1}d_i +\mathrm{{\bf a}}= \mbox{tr}(A),
% \mbox{ provided that \eqref{guanjian1} holds,}
\end{aligned}
\end{equation}
where we use  \eqref{guanjian2},   \eqref{largest-eigen1}, \eqref{yuan-lemma-proof1} and \eqref{yuan-lemma-proof2}.
This is a contradiction.

%In the case of ${\bf G}=\emptyset$, one knows that
%\begin{equation}
%\begin{aligned}
%\mbox{tr}(A)
%=\,&\lambda_n+\sum_{\alpha\in  {\bf G}}\lambda_{\alpha}+\sum_{\alpha\in {\bf B}}\lambda_{\alpha}\\
%> \,&
%(|{\bf B}|+1) (\mathrm{{\bf a}}-\frac{1}{\epsilon}\sum_{i=1}^{n-1} |a_i|^2 )-|{\bf G}| (\sum_{i=1}^{n-1}|d_i|+\epsilon )\\
%\geq
%\mathrm{{\bf a}}+
%(n-1) (\mathrm{{\bf a}}-\frac{1}{\epsilon}\sum_{i=1}^{n-1} |a_i|^2 )
 %>   \sum_{i=1}^{n-1}d_i +\mathrm{{\bf a}}= \mbox{tr}(A).
%% \mbox{ provided that \eqref{guanjian1} holds,}
%\end{aligned}
%\end{equation}
%Again, it is a contradiction.

We now prove ${\bf B}=\emptyset$.
Therefore,   ${\bf G}=I$ and  the proof is complete.
\end{proof}

We consequently obtain

\begin{lemma}
\label{refinement111}
Let $A(\mathrm{{\bf a}})$ be an $n\times n$ Hermitian  matrix
\begin{equation}
A(\mathrm{{\bf a}})=\left(
\begin{matrix}
d_1&&  &&a_{1}\\
&d_2&& &a_2\\
&&\ddots&&\vdots \\
&& &  d_{n-1}& a_{n-1}\\
\bar a_1&\bar a_2&\cdots& \bar a_{n-1}& \mathrm{{\bf a}} \nonumber
\end{matrix}
\right)
\end{equation}
with
% $d_1,\cdots, d_{n-1}$  fixed and
 $d_1,\cdots, d_{n-1}, a_1,\cdots, a_{n-1}$ fixed, and with $\mathrm{{\bf a}}$ variable.
 Assume that $d_1, d_2, \cdots, d_{n-1} $
 are distinct each other, i.e. $d_{i}\neq d_{j}, \forall i\neq j$.
Denote
$\lambda=(\lambda_1,\cdots, \lambda_n)$ by the the eigenvalues of $A(\mathrm{{\bf a}})$.
Given a positive constant $\epsilon$ with
$0<\epsilon \leq \frac{1}{2}\min\{|d_i-d_j|: \forall i\neq j\}.$
   %where $$\mu_0=\frac{1}{2}\min\{|d_i-d_j|: \forall i\neq j\}.$$
If the parameter  $\mathrm{{\bf a}}$ satisfies the quadratic growth condition % \eqref{guanjian1},
\begin{equation}
 \begin{aligned}
\label{guanjian1}
\mathrm{{\bf a}}\geq \frac{1}{\epsilon}\sum_{i=1}^{n-1}|a_i|^2 +(n-1)\sum_{i=1}^{n-1} |d_i|+(n-2)\epsilon,
\end{aligned}
\end{equation}
then the eigenvalues behavior like
% and any index $1\leq \alpha\leq n-1$, there is an index
%$i_{\alpha}$, $1\leq i_{\alpha}\leq n-1$, satisfying
\begin{equation}
\begin{aligned}
 \,& |d_{\alpha}-\lambda_{\alpha}|
<   \epsilon, \forall 1\leq \alpha\leq n-1,\\
\,& 0\leq \lambda_{n}-\mathrm{{\bf a}}
<  (n-1)\epsilon.  \nonumber
\end{aligned}
\end{equation}
\end{lemma}

\begin{proof}
The proof is based on Lemma  \ref{refinement} and a deformation argument.
Without loss of generality,  we assume $n\geq 3$ and  $\sum_{i=1}^{n-1} |a_i|^2>0$
 (otherwise  $n=2$ or the matrix $A(\mathrm{{\bf a}})$ is diagonal, and then we are done).
 Moreover, we assume in addition that $d_1<d_2\cdots<d_{n-1}$ and the eigenvalues have the order
 $$\lambda_1\leq \lambda_2\cdots \leq \lambda_{n-1}\leq \lambda_n.$$

  Fix $\epsilon\in (0,\mu_0]$, where $\mu_0=\frac{1}{2}\min\{|d_i-d_j|: \forall i\neq j\}$.
  We denote
   $$I_{i}=(d_i-\epsilon,d_i+\epsilon)$$
     and $$P_0=\frac{1}{\epsilon}\sum_{i=1}^{n-1} |a_i|^2+ (n-1)\sum_{i=1}^{n-1} |d_i|+ (n-2)\epsilon.$$
    Since $0<\epsilon \leq \mu_0$,  the intervals disjoint each other
  \begin{equation}
\label{daqin122}
%[d_{i}-\epsilon,d_{i}+\epsilon]\cap [d_{j}-\epsilon,d_{j}+\epsilon]
I_\alpha\bigcap I_\beta=\emptyset \mbox{ for }  1\leq \alpha<\beta\leq n-1.
\end{equation}

In what follows,
 we assume that  the parameter  $\mathrm{{\bf a}}$ satisfies \eqref{guanjian1} and  the Greek letters $\alpha,  \beta$
  range from $1$ to $n-1$.
Let
$$\mathrm{{\bf Card}}_\alpha: [P_0,+\infty)\rightarrow \mathbb{N}$$
% as follows:
% $$\mathrm{{\bf Card}}_\alpha=|\{ 1\leq i \leq n: \lambda_i \in I_\alpha\}|.$$
be the function that counts   the eigenvalues which lie in $I_\alpha$.
 (Note that when the eigenvalues are not distinct,  the function $\mathrm{{\bf Card}}_\alpha$ means the summation of all the %(algebraic) 
 multiplicities of  distinct eigenvalues
which  lie in $I_\alpha$).
  This function measures the number of the  eigenvalues which lie in $I_\alpha$.

 We are going to prove that $\mathrm{{\bf Card}}_\alpha$ is continuous on $[P_0,+\infty)$ in an attempt to
 complete the proof.

 Firstly, Lemma \ref{refinement} asserts that if $\mathrm{{\bf a}} \geq P_0$, then
 %  there is a $d_{i_{\alpha}}$ such that
 \begin{equation}
\label{daqin111}
\begin{aligned}
\lambda_{\alpha}\in  \bigcup_{i=1}^{n-1} I_i, \mbox{   } \forall 1\leq \alpha\leq n-1.
\end{aligned}
\end{equation}

 It is well known that the largest eigenvalue $\lambda_n\geq \mathrm{{\bf a}}$,
 while the smallest eigenvalue $\lambda_1 \leq d_1$.
%Next, let's verify $\lambda_n>\sum_{\alpha=1}^{n-1}|d_\alpha| +\epsilon$. %  for any $1\leq  \beta\leq n-1$.
 Combining it with    \eqref{daqin111}
one has
 \begin{equation}
 \label{largest1}
\begin{aligned}
\lambda_{n} \geq  \mathrm{{\bf a}}
%+ \sum_{\alpha=1}^{n-1} (d_\alpha-d_{i_\alpha}) +\sum_{\alpha=1}^{n-1}(d_{i_\alpha}-\lambda_\alpha) \\
%> \,& \mathrm{{\bf a}} -(n-1) \sum_{\alpha=1}^{n-1} | d_\alpha| - (n-1)\epsilon \\
>\,& \sum_{i=1}^{n-1}|d_i| +\epsilon,
\end{aligned}
\end{equation}
%here we also use the fact that any diagonal element is
% always less than or equals to   the  largest eigenvalue $\lambda_n\geq \mathrm{{\bf a}}$.
Thus $\lambda_n\in \mathbb{R}\setminus (\bigcup_{i=1}^{n-1} \overline{I_i})$
where $\overline{I_i}$ denotes the closure of $I_i$.
%Furthermore,
%\begin{equation}\label{daqin-yuan1}
%\sum_{\alpha=1}^{n-1}\mathrm{{\bf Card}}_\alpha(\mathrm{{\bf a}})=n-1.
%\end{equation}
%Since \eqref{largest1} implies that,  for any $1\leq \alpha\leq n-1$,  $\lambda_\alpha<\lambda_n$,
Therefore, %for any  $1\leq \alpha\leq n-1$,
the function $\mathrm{{\bf Card}}_\alpha$
is continuous (and so it is constant), since  %$\lambda_{j}\in  \bigcup_{i=1}^{n-1} I_i$ $(\forall 1\leq j\leq n-1)$,
 \eqref{daqin111}, \eqref{daqin122},  $\lambda_n\in \mathbb{R}\setminus (\overline{\bigcup_{i=1}^{n-1} I_i})$
and  the eigenvalues of $A(\mathrm{{\bf a}})$   depend  on the parameter  $\mathrm{{\bf a}}$ continuously.

The  continuity   of $\mathrm{{\bf Card}}_\alpha(\mathrm{{\bf a}})$ plays a crucial role in this proof.
Following the line of the proof Lemma 1.2  of Caffarelli-Nirenberg-Spruck \cite{CNS3} (i.e. Lemma \ref{lemmaCNS3} above),
%we know that the proof given there works in the setting of Hermitian matrices and  so
in the setting of Hermitian matrices,  one can show   that  for $1\leq \alpha\leq n-1$,
\begin{equation}
\label{yuanrr-lemma-12321}
\lim_{\mathrm{{\bf a}}\rightarrow +\infty} \mathrm{{\bf Card}}_{\alpha}(\mathrm{{\bf a}}) \geq 1.
\end{equation}
%In the proof of the above formula,  we also use \eqref{daqin122},
% $\lambda_{\alpha}\in  \bigcup_{i=1}^{n-1} I_i$ ($\forall 1\leq \alpha\leq n-1$) and
%   $\lambda_n\in \mathbb{R}\setminus (\overline{\bigcup_{i=1}^{n-1} I_i})$.
%Therefore, we have
%$$\mathrm{{\bf Card}}_{\alpha}(\mathrm{{\bf a}}) \geq 1,  \mbox{   } \forall \mathrm{{\bf a}}\in [P_0, +\infty),  \mbox{   } 1\leq \alpha\leq n-1. $$
%Combining it with    the formula $\sum_{\alpha=1}^{n-1} \mathrm{{\bf Card}}_\alpha(\mathrm{{\bf a}})=n-1$,  we know that
It follows from \eqref{largest1},  \eqref{yuanrr-lemma-12321}
 and the  continuity  of $\mathrm{{\bf Card}}_\alpha $ that
  \begin{equation}
\label{daqin142224}
\mathrm{{\bf Card}}_{\alpha}(\mathrm{{\bf a}})= 1,
 \mbox{   } \forall \mathrm{{\bf a}}\in [P_0, +\infty),  \mbox{   } 1\leq \alpha\leq n-1.  \nonumber
\end{equation}
Together with \eqref{daqin111}, we prove that,
for any   $1\leq \alpha\leq n-1$, the interval $I_\alpha=(d_{\alpha}-\epsilon,d_{\alpha}+\epsilon)$
contains the eigenvalue $\lambda_\alpha$.
We thus complete  the proof of Lemma \ref{refinement111}.
\end{proof}

%We also improve the above lemma, and show that

%\begin{lemma}
%\label{lemma-improve-yuan}
%Let $A(a)$ be the matrix in Lemma \ref{refinement111}, and we also assume  $d_1<d_2<\cdots< d_{n-1}$ and the  eigenvalues
%$\lambda=(\lambda_1,\lambda_2,\cdots, \lambda_n)$ of $A(a)$ with the order $\lambda_1\leq \lambda_2\leq \cdots \leq \lambda_n$.
%For any $\epsilon>0$, we let $a=a(\epsilon)= \frac{2}{\epsilon}\sum_{i=1}^{n-1}|a_i|^2 +(n-1)\sum_{i=1}^{n-1} |d_i|+n \epsilon$
 % corresponding to the variable $\epsilon$. Then the eigenvalues behavior like
%\begin{equation}
%\begin{aligned}
%\,& |d_{\alpha}-\lambda_{\alpha}|
%\leq \epsilon, \forall 1\leq \alpha\leq n-1,\\
%\,&|\lambda_{n}-a|
%\leq (n-1)\epsilon.  \nonumber
%\end{aligned}
%\end{equation}
%\end{lemma}

%\begin{remark}
%\label{remark11}

 Suppose that there are two distinct index $i_0, j_0$ ($i_0\neq j_0$) such that $d_{i_{0}}= d_{j_{0}}$.
  Then the characteristic polynomial of $A$
can be rewritten as the following
\begin{equation}
\begin{aligned}
(\lambda-d_{i_{0}})\left[(\lambda-\mathrm{{\bf a}})\prod_{i \neq i_{0}} (\lambda-d_i)-|a_{i_{0}}|^{2}\prod_{j\neq j_0, j\neq i_0} (\lambda-d_j)
-\sum_{i\neq i_0}|a_i|^2\prod_{j\neq i, j\neq i_{0}} (\lambda-d_j)\right]. \nonumber
\end{aligned}
\end{equation}
So $\lambda_{i_0}=d_{i_{0}}$ is an eigenvalue of $A$ for any $a\in \mathbb{R}$.
Noticing that the following polynomial
$$(\lambda-\mathrm{{\bf a}})\prod_{i \neq i_{0}} (\lambda-d_i)
-|a_{i_{0}}|^{2}\prod_{j\neq j_0, j\neq i_0} (\lambda-d_j)
-\sum_{i\neq i_0}|a_i|^2\prod_{j\neq i, j\neq i_{0}} (\lambda-d_j)$$
 is the characteristic polynomial of the $(n-1)\times (n-1)$ Hermitian  matrix
\begin{equation}
\label{matrix1}
\left(
\begin{matrix}
d_1&&  &&&&a_{1}\\
&\ddots && &&&\vdots \\
&&\widehat{d_{i_{0}}}&& &&\widehat{a_{i_{0}}}\\
&&&\ddots&&&\vdots \\
&&& &  d_{j_{0}}&& (|a_{j_{0}}|^{2}+|a_{i_{0}}|^{2})^{\frac{1}{2}}\\
&&&& &   \ddots              \\
\bar a_1&\cdots&\widehat{ \bar a_{i_{0}}}&\cdots& (|a_{j_{0}}|^{2}+|a_{i_{0}}|^{2})^{\frac{1}{2}} &\cdots & \mathrm{{\bf a}}   \nonumber
\end{matrix}
\right) 
\end{equation}
where $\widehat{*}$ indicates deletion.
Therefore, $(\lambda_1, \cdots, \widehat{\lambda_{i_{0}}}, \cdots, \lambda_{n})$ are the eigenvalues of the above $(n-1)\times (n-1)$
Hermitian  matrix.
%\end{remark}
 %Using Lemma \ref{refinement111} and Remark \ref{remark11}, one can show that
%\begin{lemma}
%\label{refinement2}
%Let $A$ be a $n\times n$ Hermitian matrix
%\begin{equation}
%\label{matrix2}
%\left(
%\begin{matrix}
%d_1&&  &&a_{1}\\
%&d_2&& &a_2\\
%&&\ddots&&\vdots \\
%&& &  d_{n-1}& a_{n-1}\\
%\bar a_1&\bar a_2&\cdots& \bar a_{n-1}& a \nonumber
%\end{matrix}
%\right)
%\end{equation}
%with $d_1,\cdots, d_{n-1}, a_1,\cdots, a_{n-1}$ fixed. Let
%\begin{equation}
%\mathcal{I}=
%\left\{
%\begin{aligned}
%\,& \mathbb{R}^{+}=(0,+\infty) \,& \mbox{ if } d_{i}=d_{1}, \forall 2\leq i \leq n-1,\\
%\,& \left(0,\mu_0\right)  \mbox{ for } \mu_0=\frac{1}{2}\min\{|d_{i}-d_{j}|: d_{i}\neq d_{j}\} \,& \mbox{ otherwise.}\nonumber
%\end{aligned}
%\right.
%\end{equation}
%Denote
%$\lambda=(\lambda_1,\cdots, \lambda_n)$ by the the eigenvalues of $A$.
%Suppose $\epsilon\in \mathcal{I}$.  Then the eigenvalues behavior like
%% and any index $1\leq \alpha\leq n-1$, there is an index
%%$i_{\alpha}$, $1\leq i_{\alpha}\leq n-1$, satisfying
%\begin{equation}
%\begin{aligned}
%\,& |d_{\alpha}-\lambda_{\alpha}|
%\leq \epsilon, \forall 1\leq \alpha\leq n-1,\\
%\,&|\lambda_{n}-a|
%\leq (n-1)\epsilon, \nonumber
%\end{aligned}
%\end{equation}
%provided that
%\begin{equation}
%|a|\geq \frac{2}{\epsilon}\sum_{i=1}^{n-1}|a_i|^2 +(n-1)\sum_{i=1}^{n-1} |d_i|+n \epsilon. \nonumber
% \end{equation}
%\end{lemma}
Hence,  we obtain %can apply Lemma \ref{refinement111} to prove
\begin{lemma}
\label{refinement3}
Let $A$ be an $n\times n$ Hermitian matrix
\begin{equation}
\label{matrix2}
\left(
\begin{matrix}
d_1&&  &&a_{1}\\
&d_2&& &a_2\\
&&\ddots&&\vdots \\
&& &  d_{n-1}& a_{n-1}\\
\bar a_1&\bar a_2&\cdots& \bar a_{n-1}& \mathrm{{\bf a}} \nonumber
\end{matrix}
\right)
\end{equation}
with $d_1,\cdots, d_{n-1}, a_1,\cdots, a_{n-1}$ fixed, and with $\mathrm{{\bf a}}$ variable.  Let  
\begin{equation}
\mathcal{I}=
%\left\{
\begin{cases}
 \mathbb{R}^{+}=(0,+\infty) \,& \mbox{ if } d_{i}=d_{1}, \forall 2\leq i \leq n-1,\\
 \left(0,\mu_0\right)  \mbox{ for } \mu_0=\frac{1}{2}\min\{|d_{i}-d_{j}|: d_{i}\neq d_{j}\} \,& \mbox{ otherwise.}\nonumber
\end{cases}
%\right.
\end{equation}
Denote
$\lambda=(\lambda_1,\cdots, \lambda_n)$ by the the eigenvalues of $A$. Fix  $\epsilon\in \mathcal{I}$.
Suppose that  the parameter $\mathrm{{\bf a}}$ in $A$ satisfies  \eqref{guanjian1}.
% \begin{equation}
%a \geq \frac{2}{\epsilon}\sum_{i=1}^{n-1}|a_i|^2 +(n-1)\sum_{i=1}^{n-1} |d_i|+n \epsilon. \nonumber
% \end{equation}
 Then the eigenvalues behavior like
\begin{equation}
\begin{aligned}
\,& |d_{\alpha}-\lambda_{\alpha}|
< \epsilon, \mbox{    } \forall 1\leq \alpha\leq n-1,\\
\,&0 \leq \lambda_{n}-\mathrm{{\bf a}}
< (n-1)\epsilon. \nonumber
\end{aligned}
\end{equation}
\end{lemma}

Applying Lemmas \ref{refinement} and \ref{refinement3},
we  complete the proof of Lemma \ref{refinement2}  without restriction to the applicable
 scope of $\epsilon$.

\begin{proof}
[Proof of Lemma \ref{refinement2}]
We follow the outline of the proof of Lemma \ref{refinement111}.
Without loss of generality, we may assume
 $$n\geq 3, \mbox{ } \sum_{i=1}^{n-1} |a_i|^2>0,  \mbox{ }  d_1\leq d_2\leq \cdots \leq d_{n-1}  \mbox{ and } \lambda_1\leq \lambda_2 \leq \cdots \lambda_{n-1}\leq \lambda_n.$$
%With the notation in the proof of Lemma \ref{refinement111}.

%With the same notation in the proof of Lemma \ref{refinement111}.
Fix $\epsilon>0$.  Let $I'_\alpha=(d_\alpha-\frac{\epsilon}{2n-3}, d_\alpha+\frac{\epsilon}{2n-3})$ and
$$P_0'=\frac{2n-3}{\epsilon}\sum_{i=1}^{n-1} |a_i|^2+ (n-1)\sum_{i=1}^{n-1} |d_i|+ \frac{(n-2)\epsilon}{2n-3}.$$
  In what follows we assume \eqref{guanjian1-yuan} holds.
The connected components of $\bigcup_{\alpha=1}^{n-1} I_{\alpha}'$ are as in the following:
$$J_{1}=\bigcup_{\alpha=1}^{j_1} I_\alpha', \mbox{ }
J_2=\bigcup_{\alpha=j_1+1}^{j_2} I_\alpha', \mbox{ }  \cdots, J_i =\bigcup_{\alpha=j_{i-1}+1}^{j_i} I_\alpha', \mbox{ } \cdots, \mbox{ }
 J_{m} =\bigcup_{\alpha=j_{m-1}+1}^{n-1} I_\alpha'.$$
 Moreover 
   \begin{equation}
   \begin{aligned}
J_i\bigcap J_k=\emptyset, \mbox{ for }   1\leq i<k\leq m. \nonumber
\end{aligned}
\end{equation}
It  plays formally the role of \eqref{daqin122} in the proof of Lemma \ref{refinement111}.

 As in the proof of Lemma \ref{refinement111},
% we define the function with the variable $\mathrm{{\bf a}}$ as follows:
  %$$ \mathrm{{\bf \widetilde{Card}}}_k:=|\{1\leq i\leq n: \lambda_i \in J_k\}|:[P_0',+\infty)\rightarrow \mathbb{N}.$$
 we let
  $$ \mathrm{{\bf \widetilde{Card}}}_k:[P_0',+\infty)\rightarrow \mathbb{N}$$
be the function that counts the eigenvalues which lie in $J_k$.
   (Note that when the eigenvalues are not distinct,  the function $\mathrm{{\bf \widetilde{Card}}}_k$ denotes  the summation of all the %(algebraic)
    multiplicities of  distinct eigenvalues which
 lie in $J_k$).
  %This function measures the number of the  eigenvalues which lie in $J_k$.
By using  Lemma \ref{refinement} and  $$\lambda_n \geq {\bf a}\geq P_0'>\sum_{i=1}^{n-1}|d_i|+\frac{\epsilon}{2n-3}$$ we conclude that
 if the parameter $\mathrm{{\bf a}}$ satisfies the quadratic growth condition \eqref{guanjian1-yuan} then
   \begin{equation}
  \label{yuan-lemma-proof5}
  \begin{aligned}
   \,& \lambda_n \in \mathbb{R}\setminus (\bigcup_{k=1}^{n-1} \overline{I_k'})
   =\mathbb{R}\setminus (\bigcup_{i=1}^m \overline{J_i}), \\
  \,& \lambda_\alpha \in \bigcup_{i=1}^{n-1} I_{i}'=\bigcup_{i=1}^m J_{i} \mbox{ for } 1\leq\alpha\leq n-1.
%  \sum_{i=1}^m  \mathrm{{\bf \widehat{Card}}}_i(\mathrm{{\bf a}})=n-1.
  \end{aligned}
  \end{equation}
%  where we also use \eqref{largest1}.
Similarly,  $ \mathrm{{\bf \widetilde{Card}}}_i(\mathrm{{\bf a}})$ is a continuous function
 with respect to the variable $\mathrm{{\bf a}}$. So it is a constant.
 Combining it with Lemma \ref{refinement3}  we see
 that $ \mathrm{{\bf \widetilde{Card}}}_i(\mathrm{{\bf a}}) =j_i-j_{i-1}$ for sufficiently large $\mathrm{{\bf a}}$.
 Here we denote $j_0=0$ and $j_m=n-1$.
%It follows from    Lemma \ref{refinement3}, \eqref{yuan-lemma-proof5} and
%  the continuity of the function $ \mathrm{{\bf \widehat{Card}}}_i$ that
%  $\lim_{\mathrm{{\bf a}}\rightarrow +\infty} \mathrm{{\bf \widehat{Card}}}_i(\mathrm{{\bf a}})=j_i-j_{i-1}$, and so
%$\mathrm{Card}_i(a) \geq j_i-j_{i-1}$. Moreover, applying $\sum_{i=1}^m \mathrm{Card}_i(a)=n-1$  we have
%Combining it with  Lemma \ref{refinement3} and \eqref{yuan-lemma-proof5} we can derive
The constant of $ \mathrm{{\bf \widetilde{Card}}}_i$  therefore follows that
$$ \mathrm{{\bf \widetilde{Card}}}_i(\mathrm{{\bf a}})
%=\lim_{\mathrm{{\bf a}}\rightarrow +\infty} \mathrm{{\bf \widetilde{Card}}}_i(\mathrm{{\bf a}})
=j_i-j_{i-1}.$$
%Combining it with \eqref{yuan-lemma-proof5}, we
We thus know that the   $(j_i-j_{i-1})$ eigenvalues
$$\lambda_{j_{i-1}+1}, \lambda_{j_{i-1}+2}, \cdots, \lambda_{j_i}$$
lie in the connected component $J_{i}$.
Thus, for any $j_{i-1}+1\leq \gamma \leq j_i$,  we have $I_\gamma'\subset J_i$ and  $\lambda_\gamma$
   lies in the connected component $J_{i}$.
Therefore,
$$|\lambda_\gamma-d_\gamma| < \frac{(2(j_i-j_{i-1})-1) \epsilon}{2n-3}\leq \epsilon.$$
Here we also use the fact that $d_\gamma$ is midpoint of  $I_\gamma'$ and every $J_i\subset \mathbb{R}$ is an open subset.
\end{proof}

\section{Quantitative Boundary Estimates}
\label{preciseboundaryestimates}

%Based on  Lemma \ref{refinement2}, we establish the  quantitative boundary estimates 
 % the real analytic Levi flatness of boundary
% under the assumption that $\partial M$ is real analytic Levi flat.
 %Both   Lemma \ref{refinement2} and the Levi flatness of $\partial M$ are key ingredients in the proof.

Given a point $p_0\in \partial  M$. %With the above notations, let
 Let $\sigma$ be the distance function to $\partial  M$,  and let $\rho(z)=\mathrm{dist}_{\bar M}(z,p_0)$ be the distance function from $z$ to $p_0$ with respect to $\omega$, and
% Let  $\sigma$ be the distance function to $\partial  M$,
\begin{equation}
%\rho(z)=\mathrm{dist}_{\bar M}(z,p_0) \mbox{ and }
 \Omega_{\delta}=\{z\in M: \rho(z)<\delta\}. \nonumber
\end{equation}

%\vspace{0.25mm}
First of all, we derive $C^0$-estimate,   gradient  estimate on
the boundary and the boundary estimates for double tangential derivatives.
Let  $w\in C^2(\bar M)$ be the function solving
\begin{equation}
\label{supersolution}
%\left\{
\begin{aligned}
 %g^{i\bar j}\mathfrak{g}_{i\bar j}[w]=
\Delta w+ g^{i\bar j}\chi_{i\bar j}  = 0   \mbox{ in } M,  \mbox{  }
w  =\varphi    \mbox{ on } \partial  M.
\end{aligned}
%\right.
\end{equation}
%and we assume $\sigma$ is  the distance function to $\partial M$.
The maximum principle and comparison principle imply that
\begin{equation}
\label{daqindiguo1}
%\left\{
\begin{aligned}
  \underline{u}\leq u  \leq w   \mbox{ in } M.
\end{aligned}
%\right.
\end{equation}
Combining with   $ \underline{u} =u = w=\varphi  \mbox{ on } \partial   M$, we know
  that  there is a uniform constant $C$ depending only on $\sup_{\bar M} |w|$,
   $\sup_{\bar M}|\underline{u}|$, $\sup_{\partial M} |\nabla w|$ and $\sup_{\partial M} |\nabla  \underline{u}|$ such that
%\begin{equation}
%\begin{aligned}
%|u |_{C^{0}(\bar M)}+|\nabla u |_{C^{0}(\partial  M)} \leq C. \nonumber
%\end{aligned}
%\end{equation}
\begin{equation}
\label{c0-bdr-c1}
\begin{aligned}
\sup_{\bar M}|u| + \sup_{\partial M}|\nabla u |  \leq C.
\end{aligned}
\end{equation}
Moreover, we can write $u-\underline{u}=h\sigma \mbox{ in } \{z\in   M: \sigma< \delta\} \mbox{ for } 0<\delta\ll 1$
and  $h \in C^2$. Thus, we have
\begin{equation}
\label{puretangential}
\begin{aligned}
\sup_{\partial M}|\nabla^2 u(\xi_1,\xi_2)  |\leq C, \mbox{ } \forall  \xi_1, \xi_2\in T_{\partial  M}, |\xi_1|=|\xi_2|=1
\end{aligned}
\end{equation}
where $\nabla^2 u$ denotes the real Hessian of $u$.

%{\color{red}{
 %We shall remark here that Proposition \ref{Tangential-Normal derivatives} holds if $\partial M$ is real analytic Levi flat, while we can Proposition \ref{proposition-quar-yuan1} when $\partial M$ is only Levi flat.
 % which includes locally flat  (\cite{Lebl-Fern}).
% }}

 Applying Lemma \ref{refinement2} we first show that the  boundary estimates for double normal  derivatives
 can be  dominated  by a quadratic term of   the boundary estimates for tangential-normal derivatives.
 \begin{proposition}
 	\label{proposition-quar-yuan1}
 	Let $(M, J, \omega)$ be a compact Hermitian manifold with $C^2$ Levi flat  boundary,  let $\nu$ be the unit inner normal vector along the boundary. %, and  $\sigma$ be the distance function to the boundary.
 	We denote $\xi_n=\frac{1}{2}(\nu-\sqrt{-1}J\nu)$.
 	%and  $J$ be the underlying complex structure. Suppose \eqref{elliptic}, \eqref{concave}, \eqref{nondegenerate} and \eqref{existenceofsubsolution} hold.
 	%Let $\sigma$ be the distance function to the boundary, and
 	%and $T^{1,0}_{\partial M}:=T^{1,0}_{\bar M} \cap T^{\mathbb{C}}_{ \partial M}=\{\xi \in T^{1,0}_{\bar M}: d\sigma(\xi)=0\}$.
 	%$\xi_n=\frac{1}{\sqrt{2}}(\nu-\sqrt{-1}J\nu)$,
 	% and let $\{\xi_\alpha\}_{\alpha=1}^{n-1}$ be a  base of  $T^{1,0}_{\partial M, x_0}$ with $g(\xi_\alpha, \bar\xi_{\beta})=\delta_{\alpha\beta}$.
 	%Let   $\xi_n=\frac{1}{\sqrt{2}}(\nu-\sqrt{-1}J\nu)$  and \eqref{elliptic}, \eqref{concave}, \eqref{nondegenerate} and \eqref{existenceofsubsolution} hold.
 	Let $u\in C^3(M)\cap C^{2}(\bar M)$ be an admissible solution  to Dirichlet problem \eqref{mainequ1}. % $\psi\in C^0(\bar M)$. 
 	Fix a point $p_0$ at the boundary.
 	Then for   $\xi_\alpha,\xi_\beta \in T_{\partial M }\cap JT_{\partial M }$ $(\alpha,\beta=1,\cdots, n-1)$  satisfying 
 	$J\xi_\alpha=\sqrt{-1}\xi_\alpha, J\xi_\beta=\sqrt{-1}\xi_\beta$   and
 	$g(\xi_\alpha, \bar\xi_{\beta}) =\delta_{\alpha\beta}$  at $p_0\in \partial M$,
 	% $\mathfrak{\underline{g}}(\xi_{\alpha}, J\bar\xi_{\beta}) $ is diagonal,
 	we have
 	\begin{equation}
 		\label{yuan-prop1}
 		\begin{aligned}
 			\mathfrak{g}(\xi_n, J\bar \xi_n)(p_0)  \leq C(1 +  \sum_{\alpha=1}^{n-1} |\mathfrak{g}(\xi_\alpha, J\bar \xi_n)(p_0)|^2),  \nonumber
 		\end{aligned}
 	\end{equation}
 	where $C$ is a uniform positive constant  depending  only on   $|u|_{C^0(\bar M)}$,   $|\underline{u}|_{C^{2}(\bar M)}$
 	and other known data
 	(but neither on $\sup_{M}|\nabla u|$ nor on $(\delta_{\psi,f})^{-1}$).
 \end{proposition}
 
The Levi flatness of the boundary implies that $\partial  M$ can be foliated by complex analytic hypersurfaces, see
  \cite{BFfoliation}. % or Lemma 5.1 in \cite{Siu2000}.
These leaves are complex local submanifolds
%. They are inside $M$
and their complex tangent spaces are  precisely the holomorphic tangent spaces $T_{\partial M}\cap JT_{\partial M}$ to
   $\partial M$.
     We shall remark here that Frobenius theorem guarantees that the Levi flat
  real hypersurface $\Sigma$ of class $C^k$ can be foliated $C^{k-1}$ complex
   hypersurfaces, while Barrett-Fornaess \cite{BFfoliation} proved that the foliation is indeed of class $C^k$.
Namely,
   \begin{lemma}
   [\cite{BFfoliation}]
   \label{foliation1}
   Let $\Sigma$ be a $C^{k}$ ($k\geq 2$) Levi flat real hypersurface in a complex manifold %$X$
    of complex dimension $n$.
   %Then $\Sigma$ is locally foliated by complex submanifolds of complex dimension $n-1$.
   %Moreover, 
   The induced foliation is indeed  of class $C^k$.
   \end{lemma}
  Let  $\mathfrak{L}$ be the leaves on $ \partial  M$, and let
  $\mathfrak{L}_{p_{0}}$ be the leaf passing through %the point 
  $p_0\in \partial  M$.
 % then $ T_{\mathfrak{L}}=T_{\partial M}\cap JT_{\partial M}$.
Such leaves allow one to construct  special local holomorphic coordinates centered at the boundary points. To be precise,
given a point $p_0\in \partial  M$, there is always a local holomorphic
 coordinate system $w=(w_1,...,w_n)$ of $\bar M$ centered at $p_0$ such that
 $(w_1,\cdots, w_{n-1})$ is a local coordinate of $\mathfrak{L}_{p_0}$ near $p_0$,
  $g(\frac{\partial}{\partial w_\alpha},\frac{\partial}{\partial \bar w_\beta})(p_0)=\delta_{\alpha\beta}$ $(1\leq \alpha, \beta\leq n-1)$,
   and  $\mathfrak{L}_{p_{0}}$ is defined
 locally by $w_n=0$.
%  \begin{equation} \begin{aligned}
%  w_n=0. \nonumber
%  \end{aligned} \end{equation}
   By changing the coordinates $w=(w_1,\cdots,w_n)$ if necessary, one can pick
    local holomorphic coordinates $z=(z_1,\cdots, z_n)$ $(z_i=x_i+\sqrt{-1}y_{i})$ such that
  % $z(p_0)=0$,
 \begin{equation}
 \label{holocoordinate}
 \begin{aligned}
 z_\alpha= w_\alpha,
 \mbox{  } g_{i\bar j}(p_0)=\delta_{ij},
  \end{aligned}
  \end{equation}
 where $1\leq \alpha\leq n-1, 1\leq i,j\leq n.$
%We also refer to %Cao-Shaw  \cite{Cao2007Shaw}, 
%Ohsawa \cite{Ohsawa2013}, Siu \cite{Siu2000} and the monograph \cite{Lebl-Fern} for some  interesting problems on Levi flat hypersurfaces.
% posed in Lins Neto \cite{Lins-Neto}.
%
%Using the local holomorphic coordinate \eqref{holocoordinate},  we prove Proposition \ref{proposition-quar-yuan1}.
 
%Based on the foliation of  Levi flat boundary,  one can  pick a holomorphic coordinate system \eqref{holocoordinate}. 

\begin{proof}
 [Proof of Proposition  \ref{proposition-quar-yuan1}]

Given $p_0\in \partial M$.  
In what follows the discussion is done at $p_0$, and the Greek letters $\alpha, \beta$ range from $1$ to $n-1$.
For the local coordinate $z=(z_1,\cdots,z_n)$   given  by \eqref{holocoordinate},
we   assume further that $\{\underline{\mathfrak{g}}_{\alpha\bar \beta}\}$ is diagonal at $p_0$.
 % for $1\leq \alpha, \beta\leq n-1$.
It follows from  the Levi flatness of $\partial  M$ and the boundary value condition that
\begin{equation}
\label{puretangential2}
\begin{aligned}
\mathfrak{g}_{\alpha\bar \beta}=\underline{\mathfrak{g}}_{\alpha\bar \beta}
=\mathfrak{g}_{\alpha\bar \beta}[\varphi] %(=\chi_{\alpha\bar\beta}+\varphi_{\alpha\bar\beta})
\end{aligned}
\end{equation}
at $p_0$.

Firstly, we claim that there exist two uniform positive constants $\varepsilon_{0}$, $R_{0}$
 depending  on  $\mathfrak{\underline{g}}$  and $f$, such that
\begin{equation}
\label{opppp}
%\left\{
\begin{aligned}
\,& f(\mathfrak{\underline{g}}_{1\bar 1}-\varepsilon_{0}, \cdots, \underline{\mathfrak{g}}_{(n-1)\overline{(n-1)}}-\varepsilon_0, R_0)\geq \psi
%\mbox{ and }\,& (\mathfrak{\underline{g}}_{1\bar 1}-\varepsilon_{0}, \cdots, \underline{\mathfrak{g}}_{(n-1)\overline{(n-1)}}-\varepsilon_0, R_1) \in \Gamma.
\end{aligned}
%\right.
\end{equation}
and
 $(\mathfrak{\underline{g}}_{1\bar 1}-\varepsilon_{0}, \cdots, \underline{\mathfrak{g}}_{(n-1)\overline{(n-1)}}-\varepsilon_0, R_0)\in \Gamma$.

We leave the proof of \eqref{opppp} at the end of the proof of this proposition.
Next,  we apply Lemma \ref{refinement2} together with \eqref{opppp} to
 establish the quantitative boundary estimates for double normal derivative.
Let's denote
\begin{equation}
A(R)=\left(
\begin{matrix}
\mathfrak{g}_{1\bar 1}&&  &&\mathfrak{g}_{1\bar n}\\
&\mathfrak{g}_{2\bar 2}&& &\mathfrak{g}_{2\bar n}\\
&&\ddots&&\vdots \\
&& &  \mathfrak{g}_{{(n-1)}\overline{(n-1)}}& \mathfrak{g}_{(n-1)\bar n}\\
\mathfrak{g}_{n\bar 1}&\mathfrak{g}_{n\bar 2}&\cdots& \mathfrak{g}_{n \overline{(n-1)}}& R  \nonumber
\end{matrix}
\right)
\end{equation}
and
\begin{equation}
\underline{A}(R)=\left(
\begin{matrix}
\mathfrak{\underline{g}}_{1\bar 1}&&  &&\mathfrak{g}_{1\bar n}\\
&\mathfrak{\underline{g}}_{2\bar 2}&& &\mathfrak{g}_{2\bar n}\\
&&\ddots&&\vdots \\
&& &  \mathfrak{\underline{g}}_{{(n-1)} \overline{(n-1)}}& \mathfrak{g}_{(n-1)\bar n}\\
\mathfrak{g}_{n\bar 1}&\mathfrak{g}_{n\bar 2}&\cdots& \mathfrak{g}_{n\overline{(n-1)}}& R  \nonumber
\end{matrix}
\right).
\end{equation}

%{\color{red}{
  %We note that, without restriction to applicable scope of the positive parameter $\epsilon$, 
  %Lemma  \ref{refinement2}  applies to % is applicable for 
 % each $\epsilon\in \mathbb{R}^+$.
Let's pick the parameter $\epsilon$  in  Lemma \ref{refinement2} as $\epsilon=\frac{\varepsilon_0}{128}$, 
%}}, 
and let
 \begin{equation}
\begin{aligned}
 R_c=\,& \frac{128(2n-3)}{\varepsilon_0}
\sum_{\alpha=1}^{n-1} | \mathfrak{g}_{\alpha\bar n}|^2
 + (n-1)\sum_{\alpha=1}^{n-1}  | \mathfrak{\underline{g}}_{\alpha\bar \alpha}|
 % +\frac{\varepsilon_0}{256} +R_1 +\frac{(n-1)\varepsilon_0}{256}
    % +\frac{(n-2)\varepsilon_0}{256(n-1)}
    +\frac{(n-2)\varepsilon_0}{128(2n-3)} +R_0, \nonumber
\end{aligned}
\end{equation}
where  $\varepsilon_0$ and $R_0$ are the constants from \eqref{opppp}.
 %Since $R_c$  satisfies the quadratic growth condition \eqref{guanjian1-yuan},
  Lemma  \ref{refinement2}  applies to
$\underline{A}(R_c)$ and
%It follows from  Lemma  \ref{refinement2}    that 
  the eigenvalues of $\underline{A}(R_c)$ 
%(possibly with an order)
%(possibly  with a proper permutation) 
shall behavior like
\begin{equation}
\label{lemma12-yuan}
\begin{aligned}
\lambda(\underline{A}(R_c)) \in
( \mathfrak{\underline{g}}_{1\bar 1}-\frac{\varepsilon_0}{128},\cdots,
 \mathfrak{\underline{g}}_{(n-1) \overline{(n-1)}}-\frac{\varepsilon_0}{128},R_c) +\overline{ \Gamma}_n \subset \Gamma.
\end{aligned}
\end{equation}
 Applying \eqref{elliptic},  \eqref{puretangential2},  \eqref{opppp} and  \eqref{lemma12-yuan},  one hence has
\begin{equation}
\begin{aligned}
F(A(R_c))=\,&  F(\underline{A}(R_c)) \\
\geq \,&
f(\mathfrak{\underline{g}}_{1\bar 1}-\frac{\varepsilon_0}{128},\cdots,
 \mathfrak{\underline{g}}_{(n-1) \overline{(n-1)}}-\frac{\varepsilon_0}{128},R_c)
\\
> \,&
 f(\mathfrak{\underline{g}}_{1\bar 1}-\varepsilon_{0}, \cdots, \underline{\mathfrak{g}}_{(n-1)\overline{(n-1)}}-\varepsilon_0, R_c)\geq \psi. \nonumber
\end{aligned}
\end{equation}
Therefore,
 \begin{equation}
 \label{nornor}
\begin{aligned}
\mathfrak{g}_{n\bar n}(p_0) \leq R_c.  \nonumber
\end{aligned}
\end{equation}
%Recall that   Proposition \ref{Tangential-Normal derivatives}  and   \eqref{puretangential}  yield
% \begin{equation}
% \label{tangentnormal2}
%\begin{aligned}
% |\mathfrak{g}_{\alpha \bar n}(p_{0})| \leq C(1+\sup_{\bar M}|\nabla u|), \forall 1\leq \alpha \leq n-1. \nonumber
%\end{aligned}
%\end{equation}
%We then achieve our goal
%%Combining it with  \eqref{tangentnormal2}-\eqref{puretangential2}, one derives
% \begin{equation}
% \label{doublenormal-yuan1}
%\begin{aligned}
%\sup_{\partial M} |\partial\bar\partial u  |\leq C(1+\sup_{\bar M}|\nabla u|^2).     \nonumber
%\end{aligned}
%\end{equation}
%%We shall remark here that we can also use Remark \ref{lemma-n2} to prove Theorem \ref{boundaryestimatethm} for $n=2$.

We notice  that there is a difference between Lemma   \ref{refinement2}  and Lemma  \ref{refinement111}. In order to
apply  Lemma  \ref{refinement111} we shall make a perturbation so that one gets a matrix satisfying the assumptions of Lemma  \ref{refinement111}. However, the perturbation is not needed when one applies Lemma \ref{refinement2}, since Lemma  \ref{refinement2} is applicable without restriction to %applicable 
scope of the parameter $\epsilon$.

To finish the proof of  Proposition  \ref{proposition-quar-yuan1}, what is left to prove is  the key inequality \eqref{opppp}.

We propose two proofs of   \eqref{opppp} here.  
Writing %to complete the proof of Theorem \ref{boundaryestimatethm}.
\begin{equation}
\underline{B}(R)=\left(
\begin{matrix}
\mathfrak{\underline{g}}_{1\bar 1}&&  &&\mathfrak{\underline{g}}_{1\bar n}\\
& \mathfrak{\underline{g}}_{2\bar 2}&& &\mathfrak{\underline{g}}_{2\bar n}\\
&&\ddots&&\vdots \\
&& & \mathfrak{\underline{g}}_{{(n-1)} \overline{(n-1)}}& \mathfrak{\underline{g}}_{(n-1)\bar n}\\
\mathfrak{\underline{g}}_{n\bar 1}&\mathfrak{\underline{g}}_{n\bar 2}&\cdots& \mathfrak{\underline{g}}_{n\overline{(n-1)}}& R  \nonumber
\end{matrix}
\right).
\end{equation}
%where $\hat{\mathfrak{\underline{g}}}_{\alpha\bar \alpha}= \mathfrak{\underline{g}}_{\alpha\bar \alpha}-\frac{(n-1-\alpha)\epsilon}{32n}$.

The first proof is  as in the following:
%We shall remark here that \eqref{opppp} can be also obtained by   Lemma \ref{lemmaCNS3} or Lemma \ref{refinement2}.
%The subsolution $\underline{u}$ satisfies
For $R>\sup_{\partial M} | \mathfrak{\underline{g}}|$, one has
\begin{equation}
\label{yuanrr-key1}
\begin{aligned}
f(\lambda(\underline{B}(R)) )>\psi \mbox{ on }   \partial M.
\end{aligned}
\end{equation}
It follows from \eqref{elliptic}, \eqref{concave},  \eqref{yuanrr-key1} and  the openness of $\Gamma$ that
%and the convexity of the level set of $f$ that
\begin{equation}
\label{yuan-pertubation-key1}
\begin{aligned}
f(\lambda(\underline{B}(R_1))-\frac{\epsilon_0}{2}\vec{1} )>\psi \mbox{ and }
(\lambda(\underline{B}(R_1))-\frac{\epsilon_0}{2}\vec{1} )\in \Gamma,
\end{aligned}
\end{equation}
where  $\epsilon_0$ (small enough) and $R_1$ (large enough) are two uniform
positive constants depending only on $\mathfrak{\underline{g}}$ and other known data.
%for each  $i$-th standard basis vector  $e_i$, $ i=1,2,\cdots, n,$
%****************************************
%and
%$$f(\lambda(\mathfrak{\underline{g}})- \frac{\epsilon_0 }{2} \vec{1} +R^0 e_i)>\psi \mbox{ and }
%(\lambda(\mathfrak{\underline{g}})- \frac{\epsilon_0 }{2} \vec{1} +R^0 e_i)\in \Gamma$$
%  for each $ i=1,2,\cdots, n,$ where $e_i$ is the $i$-th standard basis vector.  In particular,
% for any  $R\geq \mathfrak{\underline{g}}_{n\bar n}$, one has
% \begin{equation}
% %\label{lemma1233-yuan}
% \begin{aligned}
% f(\lambda(\underline{B}(R))-  \frac{\epsilon_0 }{2} \vec{1} +R^0 e_n)>\psi.
% \end{aligned}
% \end{equation}
%Let $\epsilon_0$ and $R^1$ be  uniformly positive  constants depending on $\mathfrak{\underline{g}}$
%such that  one can apply Lemma \ref{lemmaCNS3}
% %By setting $R^{1}=O(\frac{1}{\epsilon_0}|\lambda(\mathfrak{\underline{g}})|^2)$  we apply  Lemma  \ref{refinement2}
% (or  Lemma \ref{lemmaCNS3})
% to derive
% %$$\lambda(\underline{B}(R^1))\in (\mathfrak{\underline{g}}_{1\bar 1}-\frac{\epsilon_0}{2}, \cdots, \mathfrak{\underline{g}}_{(n-1)\overline{(n-1)}}-\frac{\epsilon_0}{2}, R^1-\frac{(n-1)\epsilon_0}{2})+\bar\Gamma_n \subseteq \Gamma.$$
% We also use  Lemma \ref{lemmaCNS3} to derive the above inequality (we may not obtain the constant $R^1$ explicitly).
% %Combining it with \eqref{lemma1233-yuan}  we can also derive \eqref{opppp}.
% *************************************************
Moreover, by applying Lemma  \ref{refinement2}   to the matrix $\underline{B}(R)$
 (by setting the parameter $\epsilon=\frac{\epsilon_0}{128} $ in Lemma  \ref{refinement2}),
 we know that  
 the eigenvalues
 $ \lambda(\underline{B}(R_2))=(\mu_1,\cdots, \mu_n)$ % ($R_{2}=O(\frac{128(2n-3)}{\epsilon_0}|\mathfrak{\underline{g}}|^2)$)
 ($R_{2}=\frac{128(2n-3)}{\epsilon_0}|\mathfrak{\underline{g}}|^2+(n-1)^2|\mathfrak{\underline{g}}| + \frac{(n-2)\epsilon_0}{128(2n-3)}+R_1$) 
  behavior as
\begin{equation}
\begin{aligned}
 \mathfrak{\underline{g}}_{\alpha\bar \alpha}-\frac{\epsilon_0}{128} < \mu_\alpha <  \mathfrak{\underline{g}}_{\alpha\bar \alpha}+\frac{\epsilon_0}{128},  \mbox{  }
R_2  \leq \mu_n < R_2+\frac{(n-1)\epsilon_0}{128}.
\end{aligned}
\end{equation}
Combining it with \eqref{yuan-pertubation-key1} we can prove \eqref{opppp} holds with
   $\varepsilon_0=\frac{63}{128}\epsilon_0 \mbox{ and } R_0=R_2+\frac{(n-65)\epsilon_0}{128},$
  %{\color{red}{
   where $\epsilon_0$ is the constant from \eqref{yuan-pertubation-key1}. 
   %}}.
 % In this proof, we do not use the concavity of $f$,  while  the convexity of level sets of $f$ is a ingredient.
 
We shall point out that  in this proof  condition \eqref{concave} may be replaced by the convexity of the level sets of $f$.
Moreover,
condition \eqref{yuanrr-key1}   can be also derived from
\begin{equation}
%\label{yuanrr-key1}
\begin{aligned}
\lim_{R\rightarrow +\infty }f(\lambda(\underline{B}(R)) )>\psi \mbox{ on }   \partial M.  \nonumber
\end{aligned}
\end{equation}
This condition
can be achieved by
  the boundary data  $\varphi$ according to   Lemma \ref{refinement2} and \eqref{puretangential2}.
  Also, this condition is satisfied by a $\mathcal{C}$-subsolution $\underline{u}$ with the same boundary
   value condition $\underline{u}|_{\partial M}=\varphi$.

The second proof is the following:
%Next, we give a different proof of the above  inequality \eqref{opppp}.
Applying  Lemma \ref{refinement2}
 to $\underline{B}(R)$ %(by setting $0<\epsilon\ll 1$ in Lemma \ref{refinement2})
%(possibly using a perturbation argument)
% \begin{equation}
%\underline{B}(R)=\left(
%\begin{matrix}
%\mathfrak{\underline{g}}_{1\bar 1}&&  &&\mathfrak{\underline{g}}_{1\bar n}\\
%&\mathfrak{\underline{g}}_{2\bar 2}&& &\mathfrak{\underline{g}}_{2\bar n}\\
%&&\ddots&&\vdots \\
%&& &  \mathfrak{\underline{g}}_{{(n-1)} \overline{(n-1)}}& \mathfrak{\underline{g}}_{(n-1)\bar n}\\
%\mathfrak{\underline{g}}_{n\bar 1}&\mathfrak{\underline{g}}_{n\bar 2}&\cdots& \mathfrak{\underline{g}}_{n\overline{(n-1)}}& R  \nonumber
%\end{matrix}
%\right)
%\end{equation}
we can prove that there is a uniform positive constant $R_3$ depending on $\mathfrak{\underline{g}}$
 such that $$(\mathfrak{\underline{g}}_{1\bar 1}, \cdots, \mathfrak{\underline{g}}_{(n-1)\overline{(n-1)}},R_3)\in \Gamma.$$
% for some uniformly small  constant $\epsilon>0$.
Here we also use the fact that $\Gamma$ is an open set.
The ellipticity and concavity of equation \eqref{mainequ1}, couple with Lemma 6.2 in \cite{CNS3}, therefore yield that
$$F(A)-F(B)\geq F^{i\bar j}(A)(a_{i\bar j}-b_{i\bar j})$$
for  the Hessian matrices $A=\{a_{i\bar j} \}$ and $B=\{b_{i\bar j}\}$ with $\lambda(A)$, $\lambda(B)\in \Gamma$.
%This can be  obtained by using  Lemma 6.2 in \cite{CNS3} and condition \eqref{concave}.
Thus, there exists a uniform positive constant $R_4\geq R_3$ depending only on $\mathfrak{\underline{g}}$   such that
\begin{equation}
\begin{aligned}
f(\mathfrak{\underline{g}}_{1\bar 1}, \cdots, \underline{\mathfrak{g}}_{(n-1)\overline{(n-1)}}, R_4)
=\,&F(\mbox{diag}(\mathfrak{\underline{g}}_{1\bar 1}, \cdots, \underline{\mathfrak{g}}_{(n-1)\overline{(n-1)}}, R_4))
\\
>\,&F(\underline{\mathfrak{g}})\geq \psi. \nonumber
\end{aligned}
\end{equation}
%Therefore,
% $$((\mathfrak{\underline{g}}_{1\bar 1}, \cdots, \underline{\mathfrak{g}}_{(n-1)\overline{(n-1)}}, R_4)-\delta \vec{1}+\Gamma_n)\cap \partial\Gamma^{\psi}$$
%is bounded for some  uniformly positive constant $\delta$,
 Thus one can derive \eqref{opppp} holds.
We thus complete the proof of  Proposition  \ref{proposition-quar-yuan1}.
%We can also apply Lemma \ref{refinement111} to  prove Theorem \ref{boundaryestimatethm}.

Moreover, we shall point out that one can also apply Lemma \ref{refinement111}
 to  prove  Proposition  \ref{proposition-quar-yuan1}. In order to apply  Lemma  \ref{refinement111}  we  shall make a perturbation:
\begin{equation}
\underline{A}(R,\varepsilon_0)=\left(
\begin{matrix}
\widetilde{\mathfrak{\underline{g}}}_{1\bar 1}&&  &&\mathfrak{g}_{1\bar n}\\
&\widetilde{\mathfrak{\underline{g}}}_{2\bar 2}&& &\mathfrak{g}_{2\bar n}\\
&&\ddots&&\vdots \\
&& & \widetilde{ \mathfrak{\underline{g}}}_{{(n-1)} \overline{(n-1)}}& \mathfrak{g}_{(n-1)\bar n}\\
\mathfrak{g}_{n\bar 1}&\mathfrak{g}_{n\bar 2}&\cdots& \mathfrak{g}_{n\overline{(n-1)}}& R  \nonumber
\end{matrix}
\right)
\end{equation}
where
$\mathfrak{\underline{g}}_{\alpha\bar\alpha}-\frac{\varepsilon_0}{32}\leq \widetilde{\mathfrak{\underline{g}}}_{\alpha\bar \alpha} \leq \mathfrak{\underline{g}}_{\alpha\bar \alpha}$ and
$|\widetilde{\mathfrak{\underline{g}}}_{\alpha\bar \alpha}-\widetilde{\mathfrak{\underline{g}}}_{\beta\bar \beta}|\geq \frac{\varepsilon_0}{64(n-1)}$
for $1\leq \alpha<\beta\leq n-1$, where $\varepsilon_0$ is the constant from \eqref{opppp}.

% $\widetilde{\mathfrak{\underline{g}}}_{\alpha\bar \alpha}= \mathfrak{\underline{g}}_{\alpha\bar \alpha}-\frac{(n-1-\alpha)\varepsilon_0}{32n}$.

\end{proof}

Next, we are going to prove 
%Proposition \ref{Tangential-Normal derivatives}.
the following proposition 
when $\partial M$ is real analytic Levi flat. %also use coordinate \eqref{holomorphic-coordinate-flat} 
%below to construct  barrier functions and show that
%prove that the boundedness of boundary estimates for tangential-normal derivatives depends linearly on  the supremum of gradient term.
%%\vspace{2mm}
%The Levi flatness of $\partial  M$ plays a crucial role in our proof of the following proposition.
%%which states that the boundary estimates for `tangential-normal' derivatives can be dominated by one
%%power of  the boundedness  for the  gradient term.

\begin{proposition}
	\label{Tangential-Normal derivatives}
	% Suppose that $(M,J,\omega)$ is a compact Hermitian manifold with Levi flat boundary $\partial M$ of class $C^2$.
	%Assume $\partial M$ is Levi flat and  $\nu$ is the unit interior  normal   direction to $\partial  M$.
	%Assume   $\nu$ is the unit interior  normal   direction to $\partial  M$.
	Assume the boundary is real analytic Levi flat.
	Let   $u \in C^3(M)\cap C^{2}(\bar M)$ be an admissible solution of Dirichlet problem \eqref{mainequ1},
	$\psi\in C^1(\bar M)$ and $\varphi\in C^3(\partial   M)$.
	Suppose, in addition to  \eqref{elliptic}-\eqref{addistruc},  
	that  there is an admissible subsolution
	$\underline{u}\in C^2(\bar M)$ to the Dirichlet problem.
	Then for any   $T\in  T_{\partial M}\cap JT_{\partial M}$ 
	with $|T|=1$,
	there exists a uniform constant $C$ depending only on
	$|\varphi|_{C^{3}(\bar M)}$,  % $|\sigma|_{C^{1,1}(\partial M)}$,
	$|\underline{u}|_{C^{2}(\bar M)}$, $\sup_{\partial M}|\nabla u|$,
	$|\psi|_{C^{1}(\bar M)}$, %$\partial M$ up to their second derivatives,
	and other known data
	(but not on $\sup_{ M}|\nabla u|$), such that
	\begin{equation}
		\label{tangential-normal}
		\begin{aligned}
			|\nabla^2 u(T,\nu)| \leq C (1+\sup_{ M}|\nabla u|).  % \nonumber
		\end{aligned}
	\end{equation}
	%where
	% $T \mathfrak{L}\subset TM$   is the tangent bundle of $\mathfrak{L}$,
	% $\nabla^2 u$ denotes the real Hessian of $u$.
	Moreover,  $C$  is independent of   $(\delta_{\psi,f})^{-1}$.
\end{proposition}

\begin{remark}
	\label{remark4.4}
	When $\varphi$ is a constant, the constant 
	$C$ in \eqref{tangential-normal} depends on %$|\varphi|_{C^{3}(\bar M)}$, 
	$|\psi|_{C^{1}(\bar M)}$, $|\underline{u}|_{C^{2}(\bar M)}$,
	$\partial M$ 
	up to second derivatives
	and other known data.
\end{remark}

\begin{remark}
 Combining these two Propositions, we   immediately derive
 Theorem \ref{boundaryestimatethm}.
\end{remark}

%Cartan's theorem  implies that the real analytic Levi flat boundary  is always locally given by $\mathfrak{Re}(z_n)=0$ in suitable holomorphic coordinates (cf. \cite{Cartan-1933,Lebl-Fern}). 
According to a theorem in \cite{Cartan-1933}   (see also \cite{Lebl-Fern}),
% when $\partial M$ is real analytic Levi flat, 
for each $p_0\in \partial M$, one can pick local holomorphic coordinates 
\begin{equation}
\label{holomorphic-coordinate-flat}
\begin{aligned}
(z_1,\cdots, z_n), \mbox{  } z_i=x_i+\sqrt{-1}y_i,  
\end{aligned}
\end{equation}
 centered at $p_0$, such that $g_{i\bar j}(p_0)=\delta_{ij}$ and 
 $\partial M$ is locally given by %of the form 
 \begin{equation}
 \label{holomorphic-coordinate-flat-Rez=0}
 \mathfrak{Re}(z_n)=0.  
 \end{equation}
 %. (cf. \cite{Lebl-Fern}). 
 % see \cite{Blocki2009Onthegeodesic}
Under local holomorphic coordinates \eqref{holomorphic-coordinate-flat}, we take the tangential operator on  boundary as
  \begin{equation}
  \label{tangential-operator12321-meng}
  D=%\pm \frac{\partial}{\partial y_n}, 
  \pm \frac{\partial}{\partial x_{\alpha}},
\pm \frac{\partial}{\partial y_{\alpha}},   1\leq \alpha \leq n-1.
\end{equation}
%}}
(Notice $D$ is just defined locally). % $D\in T_{\partial M}\cap JT_{\partial M}$).
%
%It  is noteworthy that such local holomorphic coordinate system \eqref{holomorphic-coordinate-flat} is only needed in the proof of \eqref{yuan-rr1} and so of Proposition \ref{Tangential-Normal derivatives}. Thus the proof also works if $\partial M$ is locally given by  \eqref{holomorphic-coordinate-flat-Rez=0}.
 In particular,  when $M=X\times S$ is a product of a closed complex manifold $X$ with a compact Riemann surface with boundary $S$, as in  \cite{Chen}, we simply choose $D$ as follows
$$D=\pm \frac{\partial}{\partial x_{\alpha}},
\pm \frac{\partial}{\partial y_{\alpha}},$$  where  $z'=(z_1,\cdots z_{n-1})$ is local holomorphic coordinate of $X$.
Thus our results extend the boundary estimates in
\cite{Chen,Phong2009Sturm} from complex Monge-Amp\`ere equation to more general equations. %\cite{Blocki2009Onthegeodesic,Chen,Phong2009Sturm}.

\vspace{1mm}
Next, we outline the following lemma. %(We omit the proof, as the proof is standard and straightforward).
\begin{lemma}
\label{DN}
There is a positive constant $C$ depending only on
$|\varphi|_{C^{3}(\bar M)}$, $|\chi|_{C^{1}(\bar M)}$, $\psi_{C^{1}(\bar M)}$
 and other known data (but neither on $\sup_M|\nabla u|$ nor on $(\delta_{\psi,f})^{-1}$) such that
\begin{equation}
\label{pili1}
\begin{aligned}
\left|\mathcal{L}D(u-\varphi)\right|\leq C\left(1+(1+\sup_{ M}|\nabla u|)
\sum_{i=1}^n f_i+  \sum_{i=1}^n f_i|\lambda_{i}|\right), \mbox{ in } \Omega_{\delta}  \nonumber
\end{aligned}
\end{equation}
  for some small $\delta>0$, where $\mathcal{L}$ is the linearized operator of equation \eqref{mainequ1} which is given by
\begin{equation}
\label{linear operator 2}
\begin{aligned}
\mathcal{L}v= F^{i\bar j}v_{i\bar j}
%+F^{i\bar j}\chi_{i\bar j,\zeta_{\alpha}}v_{\alpha}
% + F^{i\bar j}\chi_{i\bar j,\bar\zeta_{\alpha}}v_{\bar\alpha}
 \mbox{ for } v\in C^{2}( M).  \nonumber
\end{aligned}
\end{equation}
Here $F^{i\bar j}=F^{i\bar j}(\mathfrak{g})$.
\end{lemma}
% We omit the proof of Lemma \ref{DN}, as the proof  is standard and straightforward. %We omit the proof here.
 %The proof of Lemma \ref{DN} is standard. We omit the proof here.
 
 %{\color{red}{
 \begin{proof}
 We follow the notation used in  \cite{Guan2010Li}.
By straightforward computation, we have %outline some formulas  
%(see also  \cite{Guan2010Li})
\begin{equation}
	\begin{aligned}
		\,& (u_{x_k})_{\bar z_j}=u_{x_k\bar z_j}+\overline{\Gamma_{kj}^l} u_{\bar l}, 
		\,& (u_{y_k})_{\bar z_j}=u_{y_k\bar z_j}-{\sqrt{-1}}\overline{\Gamma_{kj}^l} u_{\bar l}, \nonumber
	\end{aligned}
\end{equation}
\begin{equation}
	\begin{aligned}
		(u_{x_k})_{z_i\bar z_j}=
		u_{x_k z_i\bar z_j}+\Gamma_{ik}^lu_{l\bar j}+\overline{\Gamma_{jk}^l} u_{i\bar l}-g^{l\bar m}R_{i\bar j k\bar m}u_l,  \nonumber
	\end{aligned}
\end{equation}
\begin{equation}
	\begin{aligned}
		(u_{y_k})_{z_i\bar z_j}=
		u_{y_k z_i\bar z_j}+\sqrt{-1}
		\left(\Gamma_{ik}^l u_{l\bar j}-\overline{\Gamma_{jk}^l} u_{i\bar l}\right)-\sqrt{-1}g^{l\bar m}R_{i\bar j k\bar m}u_l,
		\nonumber
	\end{aligned}
\end{equation}
\begin{equation}
\label{yuan-Bd2}
\begin{aligned}
u_{x_{k}z_i\bar z_j}
\,&
= u_{z_i\bar z_j x_{k}}+g^{l\bar m}R_{i\bar j k\bar m}u_{l}- T^{l}_{ik}u_{l\bar j}-\overline{T^l_{jk}}u_{i\bar l}, \\
u_{y_{k}z_i\bar z_j}
\,&
=u_{z_i\bar z_j y_{k}}+\sqrt{-1}g^{l\bar m}R_{i\bar j k\bar m}u_{l}
 -\sqrt{-1}(T^{l}_{ik}u_{l\bar j}- \overline{T^{l}_{j k}}u_{i\bar l}). \nonumber
\end{aligned}
\end{equation}

%At a fixed point we choose a unitary $B=(b_{ij})_{n\times n}$ which diagonalizes $(\mathfrak{g}_{i\bar j})$, we have 
%$F^{i\bar j} \mathfrak{g}_{l\bar j}=b^{ip}f_p \overline{{b}^{jp}} b_{l q} \lambda_q \overline{b_{jq}} $
On the other hand, differentiating equation \eqref{mainequ1}, we get

\begin{equation}
\begin{aligned}
\mathcal{L}( u_{x_k})=\,& 
 (\psi)_{x_k}-F^{i\bar j}(\chi_{i\bar j})_{x_k}+ F^{i\bar j} g^{l\bar m} R_{i\bar j k\bar m}u_{l}  \nonumber
- 2\mathfrak{Re}(F^{i\bar j}T^l_{ik} u_{l\bar j}),
 % \\=\,& 
% (\psi)_{x_k}-F^{i\bar j}(\chi_{i\bar j})_{x_k}+ F^{i\bar j} g^{l\bar m} R_{i\bar j k\bar m}u_{l} 
%\\\,& - 2\mathfrak{Re}(F^{i\bar j}T^l_{ik}\mathfrak{g}_{l\bar j})
%+2\mathfrak{Re}(F^{i\bar j}T^l_{ik}\chi_{l\bar j}), 
\\
\mathcal{L}(u_{y_k})=\,& 
(\psi)_{y_k}-F^{i\bar j}(\chi_{i\bar j})_{y_k} +\sqrt{-1}F^{i\bar j} g^{l\bar m} R_{i\bar j k\bar m}u_l
+2\mathfrak{Im}(F^{i\bar j} T^l_{ik}u_{l\bar j}).
%\leq \,& C\left(1+(1+\sup_{M} |\nabla u|)\sum_{i=1}^n f_i +\sum_{i=1}^n f_i|\lambda_i| \right). % \nonumber
\end{aligned}
\end{equation}
Together with $u_{l\bar j}=\mathfrak{g}_{l\bar j}-\chi_{l\bar j}$, one has
\begin{equation}
\begin{aligned}
|\mathcal{L}((u-\varphi)_{x_k})|, \mbox{  } |\mathcal{L}((u-\varphi)_{y_k})|
\leq  C\left(1+(1+\sup_{M} |\nabla u|)\sum_{i=1}^n f_i +\sum_{i=1}^n f_i|\lambda_i| \right). \nonumber
\end{aligned}
\end{equation}
%We thus prove Lemma \ref{DN}.
  %See also \cite{Guan2015Sun} for the related formula and computation.
\end{proof}
%}}

  %when the boundary is only assumed to be pseduconcave.

\begin{proof}
 [Proof of Proposition \ref{Tangential-Normal derivatives}]
The proposition is proved by constructing barrier functions. 
% To construct the desired  barrier functions, we  use  the techniques and ideas used by
% Hoffman-Rosenberg-Spruck \cite{Hoffman1992Boundary},
% Guan-Spruck  \cite{Guan1993Boundary}   %in the real setting, %and further work of
%and Guan \cite{Guan1998The,Guan12a}.
%% \cite{Guan12a,Guan14} and references therein.
The constructions of this type of  barrier functions   %originally follows from
go back at least to the work \cite{Guan1993Boundary,Guan1998The}.
%%go back to the work of Hoffman-Rosenberg-Spruck \cite{Hoffman1992Boundary}, and  to the works of Guan-Spruck
%% \cite{Guan1993Boundary}   %in the real setting, %and further work of
%%and Guan \cite{Guan1998The}.
%  %studied  complex Monge-Amp\`ere equation on general bounded domains in $\mathbb{C}^n$.
% %and further developed
%% and further work  \cite{Guan12a} of Guan.
%%by  B. Guan \cite{Guan12a}
%%who dealt with a class of fully non-linear elliptic equations on the Riemannian manifolds.

Let's take
\begin{equation}
\label{ggg}
\begin{aligned}
\widetilde{\Psi}=
A_{1} \,&\sqrt{b_{1}}(\underline{u}-u)-A_{2}\sqrt{b_{1}}\rho^{2}+A_{3}\sqrt{b_{1}}(N\sigma^{2}-t\sigma)
\\ \,&
+\frac{1}{\sqrt{b_1}}\sum_{\tau<n}|(u-\varphi)_{\tau}|^2+ D(u-\varphi),
\end{aligned}
\end{equation}
where $b_{1}=1+\sup_{  M} |\nabla (u-\varphi)|^{2}+\sup_{ M} |\nabla \varphi|^{2}$.
%This construction of $\Psi$ is the complex version of the barrier function due to Guan \cite{Guan12a}.

In what follows we denote by  $\widetilde{u}=u-\varphi$. We shall point out that the constants appearing in the proof of quantitative boundary estimates,
such as $C$, $C_0$, $C_1$,   $C_1'$,  $C_2$,  $ A_1, A_2, A_3, \cdots$, etc
 %do not depend on $|\nabla u|$.
depend   neither on $|\nabla u|$ nor on $(\delta_{\psi,f})^{-1}$.
%We shall point out that the constants  %appearing
%below, such as $C$, $C_0$, $C_1$,  $C_1'$,  $C_2$, do not depend on $|\nabla u|$.

Let $\delta>0$ and $t>0$ be sufficiently small constants such that $N\delta-t\leq 0$
 (where $N$ is a positive constant sufficiently large to be determined later), $\sigma$ is $C^2$ %in $\Omega_{\delta}$ and
 and
\begin{equation}
\label{bdy1}
\begin{aligned}
 \frac{1}{4} \leq |\nabla \sigma|\leq 2,  \mbox{  }
  |\mathcal{L}\sigma | \leq   C_2\sum_{i=1}^n f_i,     \mbox{  }
  |\mathcal{L}\rho^2| \leq C_2\sum_{i=1}^n f_{i}, \mbox{ in } \Omega_{\delta}.
\end{aligned}
\end{equation}
Furthermore, we  can choose $\delta$ and $t$  small enough  such that
\begin{equation}
\label{yuanbd-11}
\begin{aligned}
\max\{|2N\delta-t|, t\}\leq \min \left\{\frac{\varepsilon}{2C_{2}}, \frac{\beta}{16\sqrt{n}C_2} \right\},
\end{aligned}
\end{equation}
where $\beta=\frac{1}{2}\min_{\bar M} \mathrm{dist}(\nu_{\lambda[\underline{u}]}, \partial \Gamma_{n})$ is the constant in \eqref{beta},
 $\varepsilon$ is the constant  in Lemma \ref{guan2014} and $C_2$ is the constant in \eqref{bdy1}.

%Differentiating equation \eqref{mainequ1}, one has
%\begin{equation}
%\begin{aligned}
%F^{i\bar j}u_{i\bar j \tau} =\psi_{\tau}-F^{i\bar j}\chi_{i\bar j\tau}, \nonumber
%\end{aligned}
%\end{equation}
%So there is a positive uniform constant $C$ such that
%\begin{equation}
%\label{fengtian}\left\{ \begin{aligned}
% |\mathcal{L}u_{\bar\tau}|\leq \,& C ( |\nabla \psi|+ \sum_{i=1}^n f_i + \sum f_i|\lambda_i| ),\\
% |\mathcal{L}u_{\tau}|\leq \,& C ( |\nabla \psi|+\sqrt{b_1}\sum_{i=1}^n f_i + \sum f_i|\lambda_i|),
%\end{aligned} \right. \end{equation}
%where we use the standard formula
%\begin{equation}
%\left\{
%\begin{aligned}
%u_{i\bar j \bar\tau}-u_{\bar\tau i\bar j}  = \bar T^{l}_{j\tau} u_{i\bar l},  \mbox{  }
%u_{i\bar j\tau}-u_{\tau i\bar j} =-g^{l\bar m} R_{i\bar j \tau \bar m}u_{l}+T^{l}_{i\tau }u_{l\bar j}.  \nonumber
%\end{aligned}
%\right.
%\end{equation}

By straightforward calculation and an elementary inequality $|a-b|^2\geq  \frac{1}{2}|a|^2- |b|^2$, one has
\begin{equation}
\label{bdygood1}
\begin{aligned}
\mathcal{L} (\sum_{\tau<n}|\widetilde{u}_{\tau}|^2  )
=
\,&
\sum_{\tau<n} \{ F^{i\bar j}(\widetilde{u}_{\tau i}
\widetilde{u}_{\bar \tau \bar j}+ \widetilde{u}_{\bar \tau i} \widetilde{u}_{\tau \bar j})
+\mathcal{L}(\widetilde{u}_{\tau})\widetilde{u}_{\bar \tau}
+\mathcal{L}(\widetilde{u}_{\bar\tau})\widetilde{u}_{\tau}\}
\\
\geq
\,&
 \frac{1}{2}\sum_{\tau<n}F^{i\bar j} \mathfrak{g}_{\bar \tau i} \mathfrak{g}_{\tau \bar j}
   -C_1'\sqrt{b_1} \sum_{i=1}^n  f_{i}|\lambda_{i}|
   -C_1' \sqrt{b_{1}}
   -C_1' b_1\sum_{i=1}^n  f_{i}.
\end{aligned}
\end{equation}
%where we use the elementary inequality $|a-b|^2\geq  \frac{1}{2}|a|^2- |b|^2$.

By Proposition 2.19 in  \cite{Guan12a}, there is an index $r$ such that
\begin{equation}
\label{beeee}
\begin{aligned}
\sum_{\tau<n} F^{i\bar j}\mathfrak{g}_{\bar\tau i}\mathfrak{g}_{\tau \bar j}\geq
  \frac{1}{4}\sum_{i\neq r} f_{i}\lambda_{i}^{2}. \nonumber
\end{aligned}
\end{equation}
By \eqref{bdy1}, \eqref{bdygood1} and Lemma \ref{DN}, we therefore arrive at the following key inequality
\begin{equation}
\label{bdycrucial}
\begin{aligned}
\mathcal{L}(\widetilde{\Psi}) \geq \,& A_1 \sqrt{b_1} \mathcal{L}(\underline{u}-u)
+\frac{1}{8\sqrt{b_1}} \sum_{i\neq r} f_{i}\lambda_{i}^{2}
+A_3 \sqrt{b_1}\mathcal{L}(N\sigma^{2}-t\sigma)
\\
\,&
 -C_1
   -C_1 \sum_{i=1}^n  f_{i}|\lambda_{i}|
-\left(A_2C_2 +C_1 \right) \sqrt{b_1} \sum_{i=1}^n  f_{i}.
\end{aligned}
\end{equation}

{\bf Case I}: If $|\nu_{\lambda }-\nu_{\underline{\lambda} }|\geq \beta$, then by Lemma \ref{guan2014} above
and   Lemma 6.2 in \cite{CNS3}, we have
\begin{equation}
\label{bbbbb3}
\begin{aligned}
\mathcal{L}(\underline{u}-u)\geq \varepsilon  (1+\sum_{i=1}^n  f_i ),
\end{aligned}
\end{equation}
where $\varepsilon$ is the positive constant in Lemma \ref{guan2014}.
From the proof of Lemma \ref{guan2014} presented in \cite{GSS14},
we can check that $\varepsilon$ is determined by $f$, $\beta$ and $\underline{\lambda}$.

Lemma \ref{asymptoticcone1}    states that $\sum_{i=1}^n f_i\lambda_i>0$ in $\Gamma$;
while the concavity of the equation implies $\sum_{i=1}^n f_i (\underline{\lambda}_i-\lambda_i)\geq f(\underline{\lambda})-f(\lambda)\geq 0$.
We can apply them to derive
\begin{equation}
\label{guanbdr}
\begin{aligned}
% \mbox{ {\bf a)}  }
\sum_{i=1}^n  f_{i}|\lambda_{i}| =\,&
 2\sum_{\lambda_i\geq 0} f_i\lambda_i-\sum_{i=1}^n f_i\lambda_i
 \leq  \frac{\epsilon}{16\sqrt{b_1}} \sum_{\lambda_i\geq 0} f_{i}\lambda_{i}^2 + \frac{16\sqrt{b_1}}{\epsilon} \sum f_{i},\\
 % \leq \,&\frac{\epsilon}{16\sqrt{b_1}} \sum_{i\neq r} f_{i}\lambda_{i}^2 + (\frac{16\sqrt{b_1}}{\epsilon} +  \sup_{\bar M}|\underline{\lambda} |)\sum_{i=1}^n  f_{i},
%\,&\mbox{ if } \lambda_r<0, \\
%% +K_1 \left(1+\sum f_i\right), \,&\mbox{ if } \lambda_r\leq 0, \\
% \mbox{ {\bf b)}  }
 \sum_{i=1}^n f_{i}|\lambda_{i}| =\,&
 \sum_{i=1}^n f_i\lambda_i-2\sum_{\lambda_i<0}f_i\lambda_i   \leq \frac{\epsilon}{16\sqrt{b_1}} \sum_{\lambda_i<0} f_{i}\lambda_{i}^2 + (\frac{16\sqrt{b_1}}{\epsilon} +   |\underline{\lambda} |)\sum_{i=1}^n  f_{i},
 %\leq \,& \frac{\epsilon}{16\sqrt{b_1}} \sum_{i\neq r} f_{i}\lambda_{i}^2 +
 %\frac{16\sqrt{b_1}}{\epsilon} \sum f_{i}, \,&\mbox{ if } \lambda_r \geq 0,
\end{aligned}
\end{equation}
where %$f_i=\frac{\partial f}{\partial \lambda_i}(\lambda)$ and
$\epsilon>0$ is a constant to be determined later.

From \eqref{bdy1} and \eqref{yuanbd-11} it is easy to see that
\begin{equation}
\mathcal{L}(N\sigma^{2}-t\sigma) \geq %-C_2|2N\sigma-t|\sum_{i=1}^n  f_{i}. \nonumber
-\frac{\varepsilon}{2}\sum_{i=1}^n  f_{i}. \nonumber
\end{equation}
Let  $\epsilon=1/C_1$ and
 $$A_1\geq A_3+ \frac{2}{\varepsilon} (16C_1^2 +A_2C_2+C_1\sup_{\bar M}|\underline{\lambda} |
  +C_1  ).$$
Then it follows from   \eqref{bdycrucial}, \eqref{bbbbb3} and \eqref{guanbdr} that
\begin{equation}
\label{yuan-rr2}
\begin{aligned}
\mathcal{L}\widetilde{\Psi} \geq
\,&
A_1\sqrt{b_1} \mathcal{L}(\underline{u}-u)
+ \frac{1}{16\sqrt{b_1}}\sum_{i\neq r} f_i \lambda_{i}^2
-C_1
\\
\,&
- (C_1+A_2C_2+ 16C_1^2+\frac{A_3 \varepsilon}{2} +\frac{C_1\sup_{\bar M}|\underline{\lambda} |}{\sqrt{b_1}} )\sqrt{b_1} \sum_{i=1}^n  f_{i}
 \geq 0, \mbox{ in } \Omega_\delta.  
 \end{aligned}
\end{equation}
%in $\Omega_{\delta}$.

{\bf Case II}: Suppose that $|\nu_{\lambda }-\nu_{\underline{\lambda} }|< \beta$ and therefore
$\nu_{\lambda }-\beta \vec{1} \in \Gamma_{n}$ and
\begin{equation}
\begin{aligned}
f_{i} \geq  \frac{\beta }{\sqrt{n}} \sum_{j=1}^n f_{j}, \mbox{ i.e., } 
F^{i\bar j}\geq \frac{\beta }{\sqrt{n}} (F^{p\bar q} g_{p\bar q})  g^{i\bar j}.   
\end{aligned}
\end{equation}
Together with \eqref{yuanbd-11} one can derive
\begin{equation}
\begin{aligned}
\mathcal{L}(N\sigma^{2}-t\sigma)
=\,&
(2N\sigma-t)\mathcal{L}\sigma+2NF^{i\bar j}\sigma_{i}\sigma_{\bar j} \nonumber
%\\
%\geq
%\,& \left(\frac{\beta N}{8\sqrt{n}} -C_2\right)\sum f_i
 \geq \frac{\beta N}{16\sqrt{n}}\sum_{i=1}^n  f_i.   
\end{aligned}
\end{equation}

As in  \cite{GSS14} we know there exist two uniform positive constants $c_0$ and $C_0$, such that
\begin{equation}
\label{bdy22}
\begin{aligned}
\sum_{i\neq r} f_i \lambda_i^2 \geq
c_0  \sum_{i=1}^n  f_i \lambda_i^2-C_0 \sum_{i=1}^n  f_{i}.    \nonumber
%\geq
%\frac{\beta c_0}{\sqrt{n}} |\lambda|^2 \sum_{i=1}^n  f_i -C_0 \sum_{i=1}^n  f_{i}.
\end{aligned}
\end{equation}
We can check that $c_0$ depends only on $\beta$ and $n$, and $C_0$ depends only on $\beta$, $n$ and $\sup_{\bar M}|\lambda[\underline{u}]|$.
 Therefore, by choosing $N\geq 32(C_0+C_1+A_2C_2)/(\beta A_3)$, one derives
 \begin{equation}
 \label{bdygood2}
\begin{aligned}
\mathcal{L}\widetilde{\Psi}
\geq
 \frac{c_0\beta}{8\sqrt{b_1 n}} |\lambda|^2 \sum_{i=1}^n  f_i
  +\frac{\beta A_3 N}{32\sqrt{n}}\sqrt{b_1}\sum_{i=1}^n  f_i
-C_1|\lambda|\sum_{i=1}^n  f_{i}
 -C_1,  \nonumber
\end{aligned}
\end{equation}
where we use $\mathcal{L}(\underline{u}-u)\geq 0$ in $M$. Moreover, Lemma \ref{lxf4},
together with Cauchy-Schwarz inequality, implies that
\begin{equation}
\label{bdy33good}
\begin{aligned}
\mathcal{L}\widetilde{\Psi}
 \geq\,&
 %\frac{c_0\vartheta_0}{4\sqrt{b_1}} |\lambda|^2 \sum f_i
  \frac{\beta A_3 N}{64\sqrt{n}}\sqrt{b_1}\sum_{i=1}^n  f_i
  -C_1
%  \\ \,&
  +  (\frac{\sqrt{c_0 \beta^2 A_3 N}}{16\sqrt{n}}-C_1 ) |\lambda| \sum_{i=1}^n  f_i,   \mbox{ in } \Omega_{\delta}.  \nonumber
\end{aligned}
\end{equation}
If  we choose  $A_3$  sufficiently large then 
\begin{equation}
\label{yuan-rr3}
\begin{aligned}
\mathcal{L}\widetilde{\Psi} \geq 0,  \mbox{ in } \Omega_{\delta}.   \nonumber
\end{aligned}
\end{equation}

%From the construction of 
The boundary is locally given by $\mathfrak{Re}(z_n)=0$ under local holomorphic coordinates  %$z=(z_1,\cdots,z_n)$
 \eqref{holomorphic-coordinate-flat},
thus  $D(u-\varphi)=0$, $(u-\varphi)_{\tau}=0$  on $\partial   M\cap \overline{\Omega}_{\delta}$
 $(\forall 1\leq \tau<n)$, for some $\delta>0$, where $D$ is  defined in  \eqref{tangential-operator12321-meng}.
Thus we can assume $0<t, \delta\ll 1$, $N\delta-t\leq 0$ such that
\begin{equation}
\label{yuan-rr1}
\begin{aligned}
\widetilde{\Psi}= \,&A_{1}\sqrt{b_{1}}(\underline{u}-u)-A_{2}\sqrt{b_{1}}\rho^{2}+A_{3}\sqrt{b_{1}}(N\sigma^{2}-t\sigma)
\\ \,&
+\frac{1}{\sqrt{b_1}}\sum_{\tau<n}|(u-\varphi)_{\tau}|^2+ D(u-\varphi)
\leq  0
\end{aligned}
\end{equation}
on  $\partial   M\cap  \overline{\Omega}_{\delta}.$
On the other hand, $\rho=\delta$ and $\underline{u}-u\leq 0$ on $M\cap \partial\Omega_{\delta}$.
 Hence, if  $A_2\gg 1$ then
 $\widetilde{\Psi}\leq 0$
   on  $M\cap \partial\Omega_{\delta}$,
 where we  use again $N\delta-t\leq 0$.
Therefore, $\widetilde{\Psi}\leq 0$ in $\Omega_{\delta}$ by applying maximum principle. 
Together with $\widetilde{\Psi}(0)=0$, one has
$\widetilde{\Psi}_\nu(0)\leq 0$.  
%{\color{red}{
Thus
\begin{equation}
\begin{aligned}
  \nabla_\nu D(u-\varphi)(0) \leq \,& - A_1\sqrt{b_1}(\underline{u}-u)_\nu(0) +A_2 \sqrt{b_1}(\rho^2)_\nu(0) \\\,&
  - A_3\sqrt{b_1}(N\sigma^2-t\sigma)_\nu(0)  
  \\ \,& 
  -\frac{1}{\sqrt{b_1}}\sum_{\tau<n}  ((u-\varphi)_\tau \nabla_\nu(u-\varphi)_{\bar\tau})(0) \\\,&
-\frac{1}{\sqrt{b_1}}\sum_{\tau<n} ( (u-\varphi)_{\bar\tau} \nabla_\nu(u-\varphi)_{ \tau})(0)   \\
=\,&  - A_1\sqrt{b_1}(\underline{u}-u)_\nu(0) +A_2 \sqrt{b_1}(\rho^2)_\nu(0) 
\\\,&
  - A_3\sqrt{b_1}(N\sigma^2-t\sigma)_\nu(0)  \\
 \leq \,& C\sqrt{b_1} (1+\sup_{\partial M}|\nabla (u-\underline{u})|)
 \leq C' (1+\sup_M |\nabla u|).
\end{aligned}
\end{equation} 
Here we use \eqref{c0-bdr-c1} and $(u-\varphi)_\tau(0)=0$ for $1\leq\tau\leq n-1$.
The above discussion also works if we take the operator as $-D$.
%}}
Therefore, we derive
%we then obtain
%Applying maximum principle Hopf to $\widetilde{\Psi}$ on $\Omega_{\delta}$($0<\delta\ll 1$), we obtain
\begin{equation}
\label{basa}
\begin{aligned}
|\nabla_{\nu}Du| \leq C(1+\sup_{M}|\nabla u|)
\end{aligned}
\end{equation}
 at $p_0$,
where $C$ is a positive constant depends only on  %$|u|_{C^0(\bar M)}$,
 $|\varphi|_{C^{2,1}(\bar M)}$,
$|\underline{u}|_{C^{2}(\bar M)}$
 $|\psi|_{C^{1}(\bar M)}$
 and other known data
 (but not on $\sup_{ M}|\nabla u|$).
Moreover,   the constant $C$ in \eqref{basa}  does not
 depend on $(\delta_{\psi,f})^{-1}$.

  %Therefore, the estimates here can be applied to degenerate equations.

\end{proof}

%\begin{remark} 
%The proof indeed applies for the complex manifold with Levi flat boundary %supporting that
%(i.e. $T_{\partial M}\cap JT_{\partial M}$ is integrable)
%supporting that there is a local holomorphic coordinate of $\bar M$ such that locally on the boundary 
%$\frac{\partial}{\partial z_\alpha}, \frac{\partial}{\partial \bar z_\alpha}\in T_{\partial M}\cap JT_{\partial M}$.
%Typical examples are as follows: $M=X\times S$ and more generally $\partial M$ is locally flat.
%\end{remark}

%\begin{remark}	We shall remark here that Proposition \ref{Tangential-Normal derivatives} holds if $\partial M$ is real analytic Levi flat, while	Proposition \ref{proposition-quar-yuan1} holds when $\partial M$ is only Levi flat. \end{remark}

\begin{remark}	Together with the boundary value condition, Theorem \ref{boundaryestimatethm} immediately	gives the quantitative boundary estimate for real Hessian\begin{equation} \label{boundaryestimate1-realhessian}	\sup_{\partial M}|\nabla^2 u|\leq C(1+\sup_{M}|\nabla u|^2).  \nonumber
\end{equation}
\end{remark}

     %\begin{remark}
   %  Following the outline of the proof of  Proposition  \ref{proposition-quar-yuan1} we can use Lemma    \ref{refinement2}
 %to show that if $\varphi\in C^3$ and $\partial M$ is a $C^3$  pseudoconcave boundary, the complex Hessian on the boundary
 %can be dominated by a quadratic term of the boundedness for the mixed (tangential-normal) derivatives of  solutions.
 %, and furthermore, we have
 
%\end{remark}

%\begin{remark}
%\label{tangent-norm-remark1}
%We shall remark here that Proposition \ref{Tangential-Normal derivatives} still holds for complex
%Monge-Amp\`ere equation on compact K\"ahler manifolds without assuming the   Levi flatness of $\partial M$.
  %For the detail please refer to  Step 2 in the proof of Lemma 7.17 of  \cite{Boucksom2012}.
   %On the other hand,  the result proved recently by Chu-Tosatti-Weinkove \cite{ChuTosattiWeinkove2016}
  % extending a result of B{\l}ocki \cite{Blocki2009Onthegeodesic}
 % shows that the weak solutions for degenerate complex Monge-Amp\`ere equations  are   actually of class $C^{1,1}$.
 % Hence, Theorem \ref{existencede}  holds for degenerate complex Monge-Amp\`ere equation on the corresponding manifolds
  %with pseudoconcave boundary, and the weak solution   is of $C^{1,1}$.

%\end{remark}

\section{Solving Equations}
\label{solvingequation}

The following  second order estimate is   essentially due to Sz\'{e}kelyhidi \cite{Gabor}.
%Following the outline of proof of Proposition 13 in \cite{Gabor},
%  we can  to check that Sz\'{e}kelyhidi's   second order estimate still holds
%  for the Dirichlet problem  without  assuming  \eqref{addistruc}.
%The reason for removing  \eqref{addistruc} is   that the right hand side of \eqref{2nd} in Lemma \ref{guan2014}
%contains not only $\varepsilon \sum_{i=1}^n f_i(\lambda)$
%but also the constant term $\varepsilon$.

\begin{theorem}
\label{globalsecond}
 Let $( M,J,\omega)$ be a compact  Hermitian manifold  of complex dimension $n\geq 2$ with  smooth boundary. Let $\psi\in C^2(M)\cap C^{1,1}(\bar M)$.
Suppose, in addition to \eqref{elliptic}-\eqref{nondegenerate},
 that there is an admissible subsolution $\underline{u}\in C^{2}(\bar M)$ obeying \eqref{existenceofsubsolution}.
Then for any admissible solution $u\in C^{4}(M)\cap C^{2}(\bar M)$ of
Dirichlet problem \eqref{mainequ1},
there exists a uniform positive constant  $C$ depending only on
  $|u|_{C^{0}(\bar M)}$, $|\psi|_{C^{2} (\bar M)}$, $|\underline{u}|_{C^{2}(\bar M)}$,
  $|\chi|_{C^{2}(\bar M)}$  and other known data such that
\begin{equation}
\label{2-se-global-Gabor}
\begin{aligned}
\sup_{ M}|\partial\bar \partial u|
\leq C(1+ \sup_{ M}|\nabla u|^{2} +\sup_{\partial   M}|\partial\bar \partial u|).
\end{aligned}
\end{equation}

\end{theorem}
\begin{remark}
Following the outline of proof of Proposition 13 in \cite{Gabor},
  we can use  Lemma \ref{guan2014}, in place of 
 %Proposition 5 in   \cite{Gabor},
  Lemma \ref{gabor'lemma},  
   to check that Sz\'{e}kelyhidi's   second order estimate still holds
  for the Dirichlet problem  without  assuming  \eqref{addistruc}.
One can further verify that the constant $C$ in \eqref{2-se-global-Gabor} does not depend on $(\delta_{\psi, f})^{-1}$.
% as well as the regularity assumption  of $\psi$ can be relaxed to $\psi\in C^2(M)\cap C^{1,1}(\bar M)$.
 \end{remark}

%\vspace{2mm}
The existence results follow  from the standard continuity method and the above estimates.
We assume $\psi, \underline{u}\in C^{\infty}(\bar M)$.
The  general case of $\underline{u}\in C^{3}(\bar M)$ and $\psi\in C^{k,\alpha}(\bar M)$ follows by approximation  process.
Let's consider  a family of Dirichlet problems:
% as follows:
\begin{equation}
\label{conti}
%\left\{
\begin{aligned}
 F(\mathfrak{g}[u^{t}]) = (1-t)F(\mathfrak{g}[\underline{u}])+t\psi  \mbox{ in } M,  \mbox{   }
u^{t}  =\varphi  \mbox{ on } \partial  M.
\end{aligned}
%\right.
\end{equation}

We set $$I=\{t\in [0,1]: \mbox{ there exists } u^{t}\in C^{4,\alpha}(\bar M) \mbox{ solving equation } \eqref{conti}\}.$$
Clearly $0\in I$ by taking $u^{0}=\underline{u}$.
 The openness of $I$ is follows from the implicit function theorem and the estimates.

We can  verify that $\underline{u}$ is the \textit{admissible} subsolution along the whole method of continuity.
Combining Theorem \ref{boundaryestimatethm} with Theorem \ref{globalsecond},
we can derive that there exists a uniform positive constant $C$
 depending not on $|\nabla u^{t}|_{C^{0}(\bar M)}$ such that
\begin{equation}
\begin{aligned}
\sup_{M}|\Delta u^{t}| \leq C(1+\sup_{M} |\nabla u^{t}|^2). \nonumber
\end{aligned}
\end{equation}
One thus applies the blow-up argument used  in \cite{Gabor} (also previously appeared in \cite{Chen,Dinew2012Kolodziej})  to
   derive the gradient estimate,
    and so $|\partial\overline{\partial} u^t|$ has a uniform  bound.
 Applying Evans-Krylov theorem \cite{Evans82,Krylov82,Krylov83},
adapted to the complex setting (cf.  \cite{TWWYEvansKrylov2015}),
and Schauder theory,  one obtains the required higher $C^{k,\alpha}$ regularities.

 We now complete the proof of  Theorem \ref{existence}.
 Moreover, Theorem  \ref{existencede} can be derived by combining 
 Theorem \ref{existence} with the method of approximation.

\section{The Dirichlet problem on compact Riemannian manifolds with  concave  boundary}
\label{RiemannDirichlet}

Let  $(M,g)$  be a compact Riemannian manifold  with  smooth boundary $\partial  M$
 whose second fundamental form $\Pi$ is non-positive. We call such $\partial M$ a \textit{concave} boundary. 
 Throughout this section we  use  the Levi-Civita connection $\nabla$ of $(M,g)$.
 %We say that the boundary is \textit{concave} %(respectively, \textit{strictly concave})
 %if the second fundamental form $\Pi$ of $\partial  M$ is nonpositive. %(respectively, negative).
 Under Levi-Civita connection, we denote the covariant derivatives as follows $$\nabla_{i}=\nabla_{e_{i}}, \quad \nabla_{ij}  =\nabla_{i} \nabla_{j}-\Gamma_{ij}^k \nabla_{k},$$ 
 where $\Gamma^{k}_{ij}$ are the Christoffel symbols.

 In this section we briefly  discuss the Dirichlet problem analogous to   \eqref{mainequ1} on such a Riemannian manifold,
\begin{equation}
\label{mainequRiemann}
%\left\{
\begin{aligned}
f(\lambda(\mathfrak{g}[u])) =\psi   \mbox{ in } M
%u \,&=\varphi \,&\mbox{ on } \partial   M,
\end{aligned}
%\right.
\end{equation}
with $u =\varphi \mbox{ on } \partial   M$,
where we denote
$\mathfrak{g}[u]=\chi+\nabla^2u$,
$\chi$ is a smooth and symmetric $(0,2)$-tensor on $\bar M$,
and $\lambda(\mathfrak{g})=(\lambda_1,\cdots,\lambda_n)$ are the
 eigenvalues of %the symmetric $(0,2)$-tensor 
 $\mathfrak{g}$ with respect to the Riemannian metric $g$.

As  in the complex setting, we  only need to derive the quantitative boundary estimates.
%By contrast with the complex setting, one can derive  the desired quantitative boundary estimates for  tangential-normal derivatives for Dirichlet problem  on Riemannian manifolds without any geometric restriction to the boundary.
%(Please refer  to Section 8 of   \cite{Gabor} for   Sz\'ekelyhidi's brief discussion of fully non-linear elliptic equations on closed Riemannian manifolds without boundary).
The following theorem states that the quantitative boundary estimates  %for second derivatives
 hold for Dirichlet problem \eqref{mainequRiemann} on the Riemannian manifold with concave boundary.

\begin{theorem}
\label{dahanwangchao1}
With the above notations.
%Let $(M,g)$ be a compact Riemannian manifold with concave boundary of class  $C^{3}$.
Let $\psi\in  C^1(M)\cap C^{0,1}(\bar M)$ and
$\varphi\in C^3(\partial  M)$.
   Suppose, in addition to   \eqref{elliptic}-\eqref{addistruc}, that
 Dirichlet problem \eqref{mainequRiemann} admits an admissible subsolution $\underline{u}\in  C^{2}(\bar M)$.
Then for any
  admissible solution $u \in C^3(M)\cap C^{2}(\bar M)$ of the Dirichlet problem, we have
    \begin{equation}
    \label{bdr-estimate-Riemann1}
    \begin{aligned}
\sup_{\partial  M}|\nabla^2 u| \leq C(1+\sup_{ M}|\nabla u|^2),
\end{aligned}
\end{equation}
where $C$ is a uniform constant depending only on
%$|\sigma|_{C^{1,1}(\partial M)}$,
$|u|_{C^0(\bar M)}$,
$\sup_{\partial M}|\nabla u| $, $|\varphi|_{C^{2,1}(\bar M)}$, $|\underline{u}|_{C^{2}(\bar M)}$,
$|\chi|_{C^{2}(\bar M)}$,
 $ |\psi|_{C^{0,1}(\bar M)}$ % $\partial M$ up to their third  derivatives
  and other known data
 (but neither on $\sup_{M}|\nabla u|$ nor on $(\delta_{\psi,f})^{-1}$).
%Moreover,  the constant $C$  does not depend on $(\delta_{\psi,f})^{-1}$.
%Moreover, if  $\partial M$ is totally geodesic, i.e. the second fundamental of $\partial M$  is vanishing, then
%the constant $C$  in \eqref{bdr-estimate-Riemann1} depends on $\partial M$ up to second order derivatives.

   \end{theorem}

   \begin{proof}
 Here we  sketch  the proof.
  Let $(e_{1}, \cdots, e_{n})$ be  smooth orthonormal local frames around $x\in \partial  M$,
  $e_n$   the interior normal to $\partial  M$ along the boundary when restricted to $\partial M$.
It is similar to  the proof of Proposition \ref{Tangential-Normal derivatives}, one may
% apply the tangential operator on the boundary used in  \cite{CNS3} %(see also\cite{Guan1998The,GuanP1997Duke,LiYY1990})
 % to construct the barrier function,  and then
   derive
 \begin{equation}
 \label{piq}
 \begin{aligned}
\sup_{\partial M}|\nabla_{\alpha n} u| \leq C(1+\sup_M |\nabla u|), 1\leq  \alpha<n.
 \end{aligned}
 \end{equation}
 % \eqref{piq} above is a quantitative version of  (4.11) in \cite{Guan12a}.
 It is analogous to the proof of  Proposition  \ref{proposition-quar-yuan1}, we only need to verify that
  \begin{equation}
  \label{uuuu}
 \begin{aligned}
(\mathfrak{g}-\mathfrak{\underline{g}})|_{\partial  M}(e_{\alpha},e_{\beta})\geq 0, \mbox{  } \forall 1\leq\alpha,\beta\leq n-1.
\end{aligned}
 \end{equation}

Next we check \eqref{uuuu}.
 Since $u-\underline{u}=0$ on $\partial  M$, we have
 \begin{equation}
 \begin{aligned}
 % \,&\nabla_{\alpha}(u-\underline{u})  =0, \\
 \nabla_{\alpha\beta}(u-\underline{u})  =-\nabla_{n}(u-\underline{u})\Pi(e_{\alpha},e_{\beta})\mbox{ }
  \mbox{ on } \partial   M,  \mbox{  }  \forall 1\leq \alpha, \beta \leq n-1. \nonumber
 \end{aligned}
 \end{equation}
From \eqref{daqindiguo1}
 we know
 $$\nabla_{n}(u-\underline{u})\geq 0\mbox{ on } \partial   M.$$
Thus   \eqref{uuuu} holds when $\partial M$ is concave.  The proof is complete.

%Next, we consider the special case that $\partial M$ is totally geodesic hypersurface of $\bar M$.
%Fix $p_0\in \partial M$. Let $(e_1(p_0),\cdots,e_{n-1}(p_0))$ be an orthonormal base of $T_{\partial M}|_{p_0}$
%and let $e_n(p_0)$ be the inward pointing unit normal vector.
%Let $\gamma_{i}: (-\epsilon, \epsilon)\rightarrow \bar M$  be the geodesic on $(\bar M,g)$ defined on a small $(-\epsilon,\epsilon)$
%with $\gamma_i(0)=p_0$ and initial  direction $\gamma_i'(0)=e_i(p_0)$ for each $1\leq i\leq n-1$.
%The assumption that $\partial M$ is   totally geodesic
%implies that  the geodesics $\gamma_i$ ($1\leq i\leq n-1$)  stay on  the boundary, i.e. $\gamma_i((-\epsilon,\epsilon))\subset \partial M$.
%This  provides local coordinates $(x_1,\cdots, x_{n-1}, y)$ of $\bar M$ near $p_0$ such that $(x_1,\cdots, x_{n-1})$ are local coordinate of %the boundary $\partial M$,  and $\frac{\partial}{\partial y}$ is the inner  unit normal vector to the boundary.

%Based on the above local coordinates, we can follow the outline of the Proposition \ref{Tangential-Normal derivatives}
%to   derive \eqref{piq} under a weaker   $C^{2}$ regularity assumption on $\partial M$.

   \end{proof}

%Applying the blow-up arguments again, one obtains  the gradient estimates. Hence we get the $C^{2}$-bound for the \textit{admissible} %solutions.  Using Evans-Krylov theorem \cite{Evans82,Krylov82}, Schauder theory, and the method of continuity and approximation, we %obtain
%Consequently,  we obtain
\begin{theorem}
\label{existenceRiemann}
Let  $(M,g)$ be a compact Riemannian manifold with    smooth concave boundary.
 %Let $0<\alpha<1$ be a constant, and let $k\geq 2$ be an interger.
Suppose, in addition to \eqref{elliptic}-\eqref{addistruc},
 that there exists an admissible subsolution $\underline{u}\in C^{3}(\bar M)$ of
 Dirichlet problem  \eqref{mainequRiemann} with
 $\psi\in C^{\infty}(\bar M)$ and $\varphi\in C^{\infty}(\partial M)$.
 Then there exists a unique smooth admissible function $u\in C^{\infty}(\bar M)$ to solve the Dirichlet problem.

\end{theorem}

\begin{theorem}
\label{existencedeRiemann}
Let $(M, g)$ be a compact Riemannian manifold  with smooth concave boundary.
Suppose, in addition to  \eqref{elliptic}, \eqref{concave} and  \eqref{addistruc},  that $f\in C^\infty(\Gamma)\cap C^0(\bar \Gamma)$.
%that $(M, g)$ is a compact Riemannian manifold  with smooth concave   boundary.
%$C^{2,\gamma}$  boundary for some $0<\gamma\leq 1$.
Assume $\varphi\in C^{2,1}(\partial  M)$, and the function $\psi\in C^{1,1}(\bar M)$   satisfies $\delta_{\psi, f}=0$.
If  Dirichlet problem \eqref{mainequRiemann}
%\begin{equation}
%\label{degenerateequa111222111}
%%\left\{
%\begin{aligned}
%f(\lambda(\mathfrak{g}[u]))\,&= \psi
%%\lambda(\mathfrak{g}[u]) \in \bar\Gamma
%\,& \mbox{ in } M,\\
%u \,&=\varphi  \,& \mbox{ on }\partial  M
%\end{aligned}
%%\right.
%\end{equation}
admits  a strictly admissible subsolution $\underline{u}\in C^{2,1}(\bar M)$,
then it  admits a weak   solution $u\in C^{1,1}(\bar M)$
 with $\lambda(\mathfrak{g}[u])\in \bar \Gamma$ with $\Delta u \in  L^{\infty}(\bar M)$.
 % Furthermore, $$|u|_{C^{1}(\bar M)}+|\Delta u|_{ L^{\infty}(\bar M)}\leq C,$$
 %where $C$ depends only on %$|\sigma|_{C^{1,1}(\partial M)}$,
 %$|\varphi|_{C^{2,1}(\bar M)}$, $|\underline{u}|_{C^{1,1}(\bar M)}$,
%  $|\chi|_{C^{1,1}(\bar M)}$,
 %$ |\psi|_{C^{1,1}(\bar M)}$   and other known data under control.

\end{theorem}

 \begin{appendix}
	\medskip

\section{A characterization of
	level sets of $f$}
\label{Appendix}

In this appendix we  present a   characterization of
level sets of $f$ which satisfies  \eqref{elliptic}, \eqref{concave} and  \eqref{addistruc}. We   prove

\begin{proposition}
\label{charac}
Suppose \eqref{elliptic} and \eqref{concave} hold, and $\sigma\in (\sup_{\partial\Gamma}f, \sup_{\Gamma}f)$. Denote by $l_\lambda=\{t\lambda: t>0\}$.
Then the following statements are equivalent.
\begin{enumerate}
\item Condition \eqref{addistruc} holds.
\item For any $\lambda\in \Gamma$, the ray $l_{\lambda}$ intersects every $\partial\Gamma^\sigma $.
\item For any $\lambda\in \Gamma$, the ray $l_\lambda$ intersects every
$\partial\Gamma^\sigma $ at a unique one point.
\end{enumerate}
%\begin{equation}
%\begin{aligned}
% \mbox{{\bf i)} }\,& \mbox{Condition \eqref{addistruc} holds.} \\
%  \mbox{{\bf ii)} }\,&  \mbox{For any $\lambda\in \Gamma$, the ray $l_{\lambda}$ intersects any
%$\partial\Gamma^\sigma $, here $\sigma\in (\sup_{\partial\Gamma}f, \sup_{\Gamma}f)$.}  \\
% \mbox{{\bf iii)} }\,&  \mbox{For any $\lambda\in \Gamma$ and $\sigma\in (\sup_{\partial\Gamma}f, \sup_{\Gamma}f)$, the ray $l_\lambda$ intersects
%$\partial\Gamma^\sigma $ at a unique
%  }\\ \,&\mbox{one point.}   \nonumber
%\end{aligned}
%\end{equation}

\end{proposition}

\begin{proposition}
Suppose that $f$ satisfies \eqref{elliptic}-\eqref{concave}. Let $\sigma\in (\sup_{\partial\Gamma}f, \sup_{\Gamma}f)$.
If there is an $\lambda^{''}\in \partial\Gamma^\sigma$ such that $\lambda^{''}\cdot \nu_{\lambda^{''}}<0$, then we have
 $\lambda'\cdot \nu_{\lambda^{'}}=0$ and the origin $0\in T_{\lambda^{'}}\partial\Gamma^\sigma$.

\end{proposition}

\begin{proof}
Let  $g: \partial\Gamma^\sigma\rightarrow \mathbb{R}$ be a function
defined by $g(\lambda)=\sum_{i=1}^n f_i(\lambda)\lambda_i.$
%\begin{equation}
%\begin{aligned}
%g(\lambda)=\sum_{i=1}^n f_i(\lambda)\lambda_i. \nonumber
%\end{aligned}
%\end{equation}
Let $c_\sigma$
be the positive constant which satisfies $f(c_\sigma \vec{1})=\sigma$.
Clearly, $g$ is continuous, and
$$g(c_\sigma \vec{1})=c_\sigma\sum_{i=1}^n f_i(c_\sigma \vec{1})>0.$$
Note that $\lambda^{''}\cdot \nu_{\lambda^{''}}<0$. So we have $\lambda^{'}\in \partial\Gamma^\sigma$ such that
$$g(\lambda^{'})=0,
\mbox{ i.e.}, \lambda'\cdot \nu_{\lambda^{'}}=0.$$
\end{proof}

\begin{corollary}
\label{luxiaofengg1}
Given $\sigma\in (\sup_{\partial\Gamma}f, \sup_{\Gamma}f)$, and we assume $f$ satisfies \eqref{elliptic} and \eqref{concave}.
If there is an $\lambda\in \partial\Gamma^\sigma$ such that $\sum_{i=1}^n f_i(\lambda)\lambda_i\leq 0$, then there exists an $\mu\in \partial\Gamma^\sigma$ such that
$$f(t\mu)<\sigma \mbox{ for any } t>0.$$
Geometrically, the ray $\{t\mu: t>0\}$ does not intersect $\partial\Gamma^\sigma$.
\end{corollary}
\begin{proof}
 Given $\sigma\in (\sup_{\partial \Gamma}f, \sup_{\Gamma} f)$.
The conditions \eqref{elliptic} and \eqref{concave} ensure that 
%if the set $\Gamma^{\sigma}$ is not empty then 
  $\partial\Gamma^{\sigma}$ is a smooth  convex hypersurface.
%So for any $\mu\in \Gamma^\sigma$ (i.e., $f(\mu)>\sigma$)
%  and $\lambda\in \partial\Gamma^\sigma$
%$$(\mu-\lambda)\cdot \nu_{\lambda}>0.$$
Let $\lambda^{'}\in \partial\Gamma^\sigma$ be the point which satisfies
$\lambda^{'}\cdot \nu_{\lambda^{'}}=0$. Then $0\in T_{\lambda^{'}}\partial\Gamma^\sigma$.
For any $\hat{\lambda}\in \bar\Gamma^\sigma$ (i.e., $f(\hat{\lambda})\geq \sigma$), one obtains
$(\hat{\lambda}-\lambda^{'})\cdot \nu_{\lambda^{'}}\geq 0$, as the level set $\partial\Gamma^\sigma$ is a smooth convex hypersurface.
Let $\epsilon>0$ be a   sufficiently small constant such that $\mu=\lambda^{'}-\epsilon \nu_{\lambda^{'}}\in \Gamma$. Then
$$(t\mu-\lambda^{'})\cdot \nu_{\lambda^{'}}<0 \mbox{ for any } t>0.$$
Hence, $f(t\mu)<\sigma \mbox{ for any } t>0.$
\end{proof}

%To derive the boundary estimates, in this article, we prove the following lemma.

The above corollary  also yields the following:
\begin{lemma}
\label{asymptoticcone1}
Suppose   \eqref{elliptic}, \eqref{concave} and \eqref{addistruc} holds. %Let $\sigma\in (\sup_{\partial \Gamma}f, \sup_{\Gamma} f)$, then
Then $ \sum_{i=1}^n f_i(\lambda)\lambda_i >0 \mbox{ in }  \Gamma.$
%\begin{equation}
%  \label{tianji}
%  \begin{aligned}
% \sum_{i=1}^n f_i(\lambda)\lambda_i >0 \mbox{ in }  \Gamma.    \nonumber
% \end{aligned}
% \end{equation}

\end{lemma}
\begin{remark}
See also  %slightly  extends
 the inequality  $(8)'$ of Caffarelli-Nirenberg-Spruck 
\cite{CNS3}. %We refer the reader to \cite{CNS3}.
%Special Lagrangian
%Condition  \eqref{addistruc} is  slightly  a weaker  condition in \cite{CNS3}:
%For every $C>0$ and every compact subset $K\subset \Gamma$ there is a number $R=R(C,K)$ such that
%\begin{equation}\label{cns-con8}\begin{aligned}
%f(R\lambda)\geq 0.
%\end{aligned}\end{equation}
\end{remark}
%\begin{proof}
% Given $\sigma\in (\sup_{\partial \Gamma}f, \sup_{\Gamma} f)$.
%The conditions \eqref{elliptic} and \eqref{concave} ensure that if the set $\Gamma^{\sigma}$ is
%not empty then   $\partial\Gamma^{\sigma}$ is a smooth  convex hypersurface.
%So for any $\mu\in \Gamma^\sigma$ (i.e., $f(\mu)>\sigma$)
%  and $\lambda\in \partial\Gamma^\sigma$
%$$(\mu-\lambda)\cdot \nu_{\lambda}>0.$$
%
% If  $$ \sum_{i=1}^n f_i(\lambda)\lambda_i\leq 0,
%\mbox{ i.e.,}
% \lambda \cdot \nu_{\lambda}\leq 0,$$
%  then $(t\lambda-\lambda)\cdot \nu_{\lambda}\leq 0$ for all $t>1$.
%  %and the ray $l_{\lambda}=\{t\lambda: t\in \mathbb{R}^+\}$  intersects the tangent plane $T_\lambda \partial\Gamma^\sigma$ of the level set $ \partial\Gamma^\sigma$  at $\lambda$.
%  Hence, $f(t\lambda)\leq \sigma$ for $t>1$. Therefore,
% $$\lim_{t\rightarrow +\infty}f(t\lambda)\leq \sigma$$
% which  contradicts with \eqref{addistruc}.
% \end{proof}

%We now   give a description of  the level sets of the function $f$  satisfying  \eqref{elliptic}, \eqref{concave} and \eqref{addistruc}.

\begin{proof}
[Proof of Proposition \ref{charac}]
Corollary \ref{luxiaofengg1} and Lemma \ref{asymptoticcone1} imply that if one of
 the above three conditions (1), (2), (3) holds  then
\begin{equation}
\label{luxiaofengg2}
\sum_{i=1}^{n} f_i(\lambda) \lambda_i>0.
\end{equation}
Hence, $f(t\lambda)$ is  strictly monotone increasing in $t\in \mathbb{R}^+$, and so (2) $\Leftrightarrow$ (3).

(1) $\Rightarrow$ (2): If there exist $\sigma\in (\sup_{\partial\Gamma}f, \sup_{\Gamma}f)$ and $\lambda\in \Gamma$, such that the ray
$\{t\lambda: t>0\}$ does not intersect $\partial\Gamma^\sigma$, then one easily derive
$f(t\lambda)<\sigma$ for any $t>0$.
It is a contradiction.

(2) $\Rightarrow$ (1): For any $\lambda\in \Gamma$  and $\sigma\in (\sup_{\partial\Gamma}f, \sup_{\Gamma}f)$,
we know that there is  $t_0>0$ such that $f(t_0\lambda)=\sigma$. So
 condition \eqref{addistruc} holds as one has \eqref{luxiaofengg2}.
\end{proof}

\end{appendix}

%\vspace{2mm}

\bigskip

%{\bf Acknowledgement.}
%The author is very grateful to  Professor Bo Guan for  introducing  him to the area of fully non-linear elliptic equations.
% The author would also wish to express his gratitude to Professor Chunhui Qiu and Professor Xi Zhang for their support  and  kindly encouragements.  Thanks is also due to  Professor Xinan Ma  for his encouragement.

\small
\bibliographystyle{plain}

\end{CJK*}
\end{document}